\documentclass[11pt,leqno]{amsart}
\usepackage{amssymb}
\usepackage{xypic}
\begin{document}
\theoremstyle{plain}
\newtheorem{thm}{Theorem}[section]
\newtheorem*{thm1}{Theorem 1}
\newtheorem*{thm2}{Theorem 2}
\newtheorem{lemma}[thm]{Lemma}
\newtheorem{lem}[thm]{Lemma}
\newtheorem{cor}[thm]{Corollary}
\newtheorem{propose}[thm]{Proposition}
\newtheorem{ex}[thm]{Example}
\theoremstyle{definition}
\newtheorem{rmk}[thm]{Remark}
\newtheorem{defn}[thm]{Definition}
\newtheorem{notations}[thm]{Notations}
\newtheorem{claim}[thm]{Claim}
\newtheorem{ass}[thm]{Assumption}
\numberwithin{equation}{section}
\newcounter{elno}      
\def\points{\list
{\hss\llap{\upshape{(\roman{elno})}}}{\usecounter{elno}}}
\let\endpoints=\endlist

%
%
%
\newcommand{\mc}{\mathcal} 
\newcommand{\mb}{\mathbb} 
\newcommand{\surj}{\twoheadrightarrow} 
\newcommand{\inj}{\hookrightarrow} \newcommand{\zar}{{\rm zar}} 
\newcommand{\an}{{\rm an}} \newcommand{\red}{{\rm red}} 
\newcommand{\Rank}{{\rm rk}} \newcommand{\codim}{{\rm codim}} 
\newcommand{\rank}{{\rm rank}} \newcommand{\Ker}{{\rm Ker \ }} 
\newcommand{\Pic}{{\rm Pic}} \newcommand{\Div}{{\rm Div}} 
\newcommand{\Hom}{{\rm Hom}} \newcommand{\im}{{\rm im}} 
\newcommand{\Spec}{{\rm Spec \,}} \newcommand{\Sing}{{\rm Sing}} 
\newcommand{\sing}{{\rm sing}} \newcommand{\reg}{{\rm reg}} 
\newcommand{\Char}{{\rm char}} \newcommand{\Tr}{{\rm Tr}} 
\newcommand{\Gal}{{\rm Gal}} \newcommand{\Min}{{\rm Min \ }} 
\newcommand{\Max}{{\rm Max \ }} \newcommand{\Alb}{{\rm Alb}\,} 
\newcommand{\GL}{{\rm GL}\,} 
\newcommand{\ie}{{\it i.e.\/},\ } \newcommand{\niso}{\not\cong} 
\newcommand{\nin}{\not\in} 
\newcommand{\soplus}[1]{\stackrel{#1}{\oplus}} 
\newcommand{\by}[1]{\stackrel{#1}{\rightarrow}} 
\newcommand{\longby}[1]{\stackrel{#1}{\longrightarrow}} 
\newcommand{\vlongby}[1]{\stackrel{#1}{\mbox{\large{$\longrightarrow$}}}} 
\newcommand{\ldownarrow}{\mbox{\Large{\Large{$\downarrow$}}}} 
\newcommand{\lsearrow}{\mbox{\Large{$\searrow$}}} 
\renewcommand{\d}{\stackrel{\mbox{\scriptsize{$\bullet$}}}{}} 
\newcommand{\dlog}{{\rm dlog}\,} 
\newcommand{\longto}{\longrightarrow} 
\newcommand{\vlongto}{\mbox{{\Large{$\longto$}}}} 
\newcommand{\limdir}[1]{{\displaystyle{\mathop{\rm lim}_{\buildrel\longrightarrow\over{#1}}}}\,} 
\newcommand{\liminv}[1]{{\displaystyle{\mathop{\rm lim}_{\buildrel\longleftarrow\over{#1}}}}\,} 
\newcommand{\norm}[1]{\mbox{$\parallel{#1}\parallel$}} 
\newcommand{\boxtensor}{{\Box\kern-9.03pt\raise1.42pt\hbox{$\times$}}} 
\newcommand{\into}{\hookrightarrow} \newcommand{\image}{{\rm image}\,} 
\newcommand{\Lie}{{\rm Lie}\,} 
\newcommand{\CM}{\rm CM}
\newcommand{\sext}{\mbox{${\mathcal E}xt\,$}} 
\newcommand{\shom}{\mbox{${\mathcal H}om\,$}} 
\newcommand{\coker}{{\rm coker}\,} 
\newcommand{\sm}{{\rm sm}} 
\newcommand{\tensor}{\otimes} 
\renewcommand{\iff}{\mbox{ $\Longleftrightarrow$ }} 
\newcommand{\supp}{{\rm supp}\,} 
\newcommand{\ext}[1]{\stackrel{#1}{\wedge}} 
\newcommand{\onto}{\mbox{$\,\>>>\hspace{-.5cm}\to\hspace{.15cm}$}} 
\newcommand{\propsubset} {\mbox{$\textstyle{ 
\subseteq_{\kern-5pt\raise-1pt\hbox{\mbox{\tiny{$/$}}}}}$}} 
\newcommand{\sA}{{\mathcal A}}
\newcommand{\sa}{{\mathcal a}} 
\newcommand{\sB}{{\mathcal B}} \newcommand{\sC}{{\mathcal C}} 
\newcommand{\sD}{{\mathcal D}} \newcommand{\sE}{{\mathcal E}} 
\newcommand{\sF}{{\mathcal F}} \newcommand{\sG}{{\mathcal G}} 
\newcommand{\sH}{{\mathcal H}} \newcommand{\sI}{{\mathcal I}} 
\newcommand{\sJ}{{\mathcal J}} \newcommand{\sK}{{\mathcal K}} 
\newcommand{\sL}{{\mathcal L}} \newcommand{\sM}{{\mathcal M}} 
\newcommand{\sN}{{\mathcal N}} \newcommand{\sO}{{\mathcal O}} 
\newcommand{\sP}{{\mathcal P}} \newcommand{\sQ}{{\mathcal Q}} 
\newcommand{\sR}{{\mathcal R}} \newcommand{\sS}{{\mathcal S}} 
\newcommand{\sT}{{\mathcal T}} \newcommand{\sU}{{\mathcal U}} 
\newcommand{\sV}{{\mathcal V}} \newcommand{\sW}{{\mathcal W}} 
\newcommand{\sX}{{\mathcal X}} \newcommand{\sY}{{\mathcal Y}} 
\newcommand{\sZ}{{\mathcal Z}} \newcommand{\ccL}{\sL} 
\newcommand{\A}{{\mathbb A}} \newcommand{\B}{{\mathbb 
B}} \newcommand{\C}{{\mathbb C}} \newcommand{\D}{{\mathbb D}} 
\newcommand{\E}{{\mathbb E}} \newcommand{\F}{{\mathbb F}} 
\newcommand{\G}{{\mathbb G}} \newcommand{\HH}{{\mathbb H}} 
\newcommand{\I}{{\mathbb I}} \newcommand{\J}{{\mathbb J}} 
\newcommand{\M}{{\mathbb M}} \newcommand{\N}{{\mathbb N}} 
\renewcommand{\P}{{\mathbb P}} \newcommand{\Q}{{\mathbb Q}} 

\newcommand{\R}{{\mathbb R}} \newcommand{\T}{{\mathbb T}} 
\newcommand{\U}{{\mathbb U}} \newcommand{\V}{{\mathbb V}} 
\newcommand{\W}{{\mathbb W}} \newcommand{\X}{{\mathbb X}} 
\newcommand{\Y}{{\mathbb Y}} \newcommand{\Z}{{\mathbb Z}} 

\title{$F$-thresholds $c^I({\bf m})$ for projective curves}
\author{Vijaylaxmi Trivedi}
\date{}
\address{School of Mathematics, Tata Institute of Fundamental Research, 
Homi Bhabha Road, Mumbai-40005, India}
\email{vija@math.tifr.res.in}
\thanks{We acknowledge support of the Department of Atomic Energy, 
Government of India, under project no. 12-R$\&$D-TFR-RTI4001.}
\begin{abstract}
We show that if $R$ is a two dimensional standard graded ring (with the 
graded maximal ideal ${\bf m}$) of characteristic $p>0$ and $I\subset R$ is a   graded ideal 
 with $\ell(R/I) <\infty$  then
 the $F$-threshold   $c^I({\bf m})$ can be expressed in terms
of a strong HN (Harder-Narasimahan) slope of  
the canonical syzygy bundle  on $\mbox{Proj}~R$. Thus $c^I({\bf m})$ is 
a rational number.

This gives us a well defined notion, of the  $F$-threshold $c^I({\bf m})$
 in characteristic $0$, 
in terms of a HN
slope of the syzygy bundle on $\mbox{Proj}~R$.

This generalizes our earlier result (in [TrW]) where we have shown 
that if  $I$ has homogeneous 
generators of the same degree, then the $F$-threshold 
$c^I({\bf m})$ 
is expressed in terms of the minimal strong HN slope 
(in char~$p$) and  in terms of the minimal 
HN slope (in char~$0$), respectively, of the canonical 
syzygy bundle on $\mbox{Proj}~R$.

Here we also  prove that, for a given pair $(R, I)$ over a  
field of characteristic $0$, if $({\bf m}_p, I_p)$ is 
a {\em reduction mod} $p$ of $({\bf m}, I)$ then  
 $c^{I_p}({\bf m}_p) \neq c^I_{\infty}({\bf m})$ implies  
$c^{I_p}({\bf m}_p)$ has $p$ in the denominator, for almost all $p$. 

\end{abstract}

\maketitle
\section{Introduction}
Let $(R, I)$ be a standard graded pair, {\em i.e.}, $R$ is a Noetherian 
standard graded ring  over a perfect field $k$ (unless otherwise stated) of 
characteristic $p >0$ 
and $I$ is a graded ideal of finite colength. 
 Let ${\bf m}$ be the graded maximal ideal of $R$.

If $M$ is a  finitely generated graded $R$-module 
then (see [T2]) we have 
a compactly supported continuous function $f_{M, I}:[0, \infty)\longto [0, \infty)$
called  the {\em HK density function}  for $(M,I)$.
  We 
realize this function as the limit of a uniformly convergent sequence of 
compactly supported functions 
$\{f_n(M, I):\R\to [0, \infty)\}_{n\in \N}$, where  
$$f_{n}(M, I)(x) = \frac{1}{q^{d-1}}\ell(M/I^{[q]}M)_{\lfloor xq\rfloor},
~~\mbox{for}~~q=p^n.$$

Moreover $$\int_0^{\infty}f_{M, I}(x)dx = 
e_{HK}(M, I),$$
 where $e_{HK}(M, I)$ denotes the invariant 
 HK multiplicity of $M$ with 
respect to $I$ (introduced by P. Monsky [M]).

Since  the function $f_{M,I}$ is the 
uniformly  convergent limit of  the sequence $\{f_n(M,I)\}_n$,
and is  also  `additive' and  `multiplicative', 
  it has proved to be  a versatile tool to handle invariants attached to it.

The focus of this paper is on the another invariant,
the {\em maximum support} of the function $f_{R,I}$, namely  
 the number $\alpha(R, I) = \mbox{Sup}~\{x\mid f_{R,I}(x)\neq 0 \}$.
Here we  consider the standard graded pair $(R,I)$, where $R$ is a two dimensional 
domain.

In the case $\dim~R\geq 2$ this invariant relates  to another well known 
invariant, the {\em $F$-threshold} 
$c^I({\bf m})$ of ${\bf m}$ with respect to $I$: 

\vspace{5pt}

\noindent{\bf Theorem}~(Theorem~4.9, [TrW]).\quad{\em Let $(R, I)$ be a standard graded pair 
and ${\bf m}$ be the graded maximal ideal of $R$. 
If $R$ is strongly $F$-regular on the punctured
spectrum (for example if $\mbox{Proj}~R$ is smooth) then  
$\alpha(R, I) = c^I({\bf m}).$

In particular when $R$ is a normal domain of dimension two then  
$\alpha(R, I) = c^I({\bf m})$.}

\vspace{5pt} 

We recall that for a pair of 
ideals $I$ and $J$, 
the  $F$-threshold of $J$ with respect to $I$ is defined as  
$$c^{I}(J) =  \lim_{q\to \infty}
\frac{\mbox{min}~\{r\mid {J}^{r+1}\subseteq {I}^{[q]}\}}{q}.$$
This was first introduced by Musta\c{t}\u{a}-Takagi-Watanabe  in 
[MTW] for regular rings, and  later, in a more general setting (when $R$ is not 
regular), was further studied by Huneke-Musta\c{t}\u{a}-Takagi-Watanabe in [HMTW].

In [TrW] we studied $\alpha(R,I)$ ($= c^I({\bf m}))$  in detail when $I$ is 
generated by homogeneous elements of the same degree. 
In this paper we generalize  the results, proved in  [TrW] for 
 the two dimensional case, 
to the case when 
$I$  has a set of homogeneous generators, but not neccessarily 
of the same degree.
 The 
technique used in  [TrW]) does not work here. 
 We elaborate on 
this now.
  
For a given pair $(R, I)$, let $S$ be the 
normalization of $R$ in the  quotient field $Q(R)$. Then $X=\mbox{Proj}~S$ is a
nonsingular curve with the ample line bundle $\sO_X(1)$.
Analogous to the notion of the HK density function $f_{R,I}$ for the pair $(R,I)$, 
 we can have  the notion of 
the HK density function $f_{V, \sO_X(1)}$  for the pair $(V, \sO_X(1))$, where 
$V$ is a vector-bundle on $X$ and $\sO_X(1)$ the ample line bundle of $X$.
The function $f_{V, \sO_X(1))}$ 
has an explicit formula in terms of the strong HN data (see Notations~\ref{hnd}) 
of $V$. Moreover
 the maximum  support of $f_{V, \sO_X(1)}$ has an explicit formula  in 
terms of the 
minimum strong HN slope (denoted by $a_{min}(V)$) of the vector bundle $V$.

We relate the function $f_{R,I}$ with the HK density functions of 
specific vector bundles on $X$ by the formula
 $$f_{R, I}(x) =  f_{V_{0}, \sO_X(1)}(x)-f_{M_{0}, \sO_X(1)}(x),$$
where if $I$ has 
a set of homogeneous generators $f_1, \ldots, f_s$ of degree 
$d_1, \ldots, d_s$ then there is  the canonical short exact
sequence  
\begin{equation}\label{e22}0\longto V_0\longto M_0 = \oplus_i\sO_X(1-d_i) 
\longto \sO_X(1)\longto 0 \end{equation}
(see the sequence~(\ref{e2}) in subsection~2.2) of locally free sheaves of 
$\sO_X$-modules.
  We recall

\vspace{5pt}
 
\noindent{\bf Theorem~6.3}~([TrW])\quad {\em If $d_1 = \cdots = d_s$ then 
 $\alpha(R,I) = 1-a_{min}(V_0)/d$.}

\vspace{5pt}

The main point here was that the bundle 
 $M_0$ is  strongly semistable and hence  
$$a_{min}(V_0) \leq \mu(V_0) < \mu(M_0) = a_{min}(M_0),$$
where $\mu(W) = \deg(W)/\rank(W)$ denotes the slope of $W$.
 This  implied   
that $$\mbox{max Supp}~f_{M_{0}, \sO_X(1)} <\mbox{max Supp}~f_{V_{0}, \sO_X(1)} 
= 1-a_{min}(V_0)/d.$$

The above formula for $\alpha(R,I)$, in terms of the strong HN data of 
$V_0$, straightaway
gave a well defined  notion of $\alpha(R, I)$ 
(hence of $c^I({\bf m})$) in characteristic $0$, as  (by Lemma~1.16~[T1]) 
$\lim_{p_s\to \infty}a_{min}(V^{s}) = \mu_{min}(V)$, where $V^{s}$ denotes the 
{\em reduction mod} $p_s$ of the  bundle $V$.

In particular,  if $(R_s, I_s)$ is the {\em reduction mod $p_s$} of 
the pair $(R, I)$ then this implied
 $\lim_{p_s\to \infty}\alpha(R_s, I_s) = 1-\mu_{min}(V_0).$

\vspace{5pt}

 However, if $I$ is not given by a homogeneous 
set of generators of the same degrees then  $M_0$ is not a strongly semistable
(or even a semistable) bundle. It may happen, as shown by an 
example given in 
 Remark~\ref{rrr}, that  
$a_{min}(V_0) = a_{min}(M_0)$ and the functions 
$f_{M_0, \sO_X(1)}$ and $f_{V_0, \sO_X(1)}$ may coincide 
in a neighbourhood of their common maximal supports.
Hence $\alpha({R,I})$ cannot have a description as in [TrW].

In this paper we circumvent this  difficulty, 
by introducing the notion of the 
$\mu$-{\em reduction bundle} and
{\em strong} $\mu$-{\em reduction
bundle}  of $V_0$ (strictly speaking, of the exact sequence (\ref{e22}) 
of vector bundles):
Consider the HN 
filtration
$$0 =M_{l_1} \subset M_{l_1-1} \subset \cdots \subset M_0$$ 
of $M_0$ 
(hence for any $m\geq 0$, 
$0 =F^{m*}M_{l_1} \subset F^{m*}M_{l_1-1} \subset \cdots \subset F^{m*}M_0 $
is the HN filtration of $F^{m*}M_0$).
Let $0\subset V_{l_1}\subset V_{l_1-1} \subset \cdots \subset V_0$ be 
 the induced 
(this need not be the HN) filtration 
on $V_0$. Then $V_t$ is the  $\mu$-{\em reduction bundle}  of $V_0$ if 
$t$ is the least integer such that $\mu_{min}(V_t) <\mu_{min}(M_t)$.
A bundle $V_{t_0}$ is the {\em strong $\mu$-reduction bundle} of $V_0$ 
if $F^{m*}(V_{t_0})$ is the $\mu$-reduction bundle of 
$F^{m*}(V_0)$, where $m$ is an integer (such an integer does exist) 
where $F^{m*}(V_0)$ has the strong HN filtration.
Moreover we show 
$$0\longto V_{t_0}\longto M_{t_0}\longto \sO_X(1)
\longto 0$$
 is a  short exact sequence of $\sO_X$-modules. 
We show (in Theorem~\ref{n1}) 

\vspace{5pt}

\noindent{\bf Theorem~A}.~~{\it If $(R, I)$ is a two dimensional 
 standard graded pair 
with the multiplicity $d = e_0(R, {\bf m})$ and
$V_{t_0}$ is the strong $\mu$-reduction bundle of $V_0$ then 
\begin{enumerate}
\item  $f_{R, I}(x) =  f_{V_{0}, \sO_X(1)}(x)-f_{M_{0}, \sO_X(1)}(x) = 
f_{V_{t_0}, \sO_X(1)}-f_{M_{t_0}, \sO_X(1)}$
and 
\item $ \alpha(R,I) = 1 - \frac{a_{min}(V_{t_0})}{d}$.
\end{enumerate}}

Though $F^{m*}V_{t_0}$  may not be one of the bundles occuring in the
HN filtration of $F^{m*}V_0$, the slope $a_{min}(V_{t_0})$ is equal to one of the
strong HN slopes of $V_0$. In particular,   
$\alpha(R, I)$ is still given in terms of the strong HN data of $V_0$.

Moreover, we show that the notion of strong $\mu$-reduction and $\mu$-reduction bundles 
behaves well 
under {\em reduction mod $p$}. This leads to a well defined notion of 
$\alpha(R, I)$ in characteristic $0$ (Lemma~\ref{l*} and Theorem~\ref{vb1})

\vspace{5pt}

\noindent{\bf Theorem~B}.~~{\it Let $(R, I)$ be a two dimensional 
 standard graded pair  in characteristic $0$ and let 
$V_0$ be the syzygy bundle  on $X$ as in the sequence (\ref{e22}). Let 
$(A, X_A, V_A)$ 
and $(A, R_A, I_A)$  
be spreads for $(X, V_0)$ and $(R, I)$, respectively.
 If
$V_t$ is the $\mu$-reduction bundle of $V_0$
then, for a closed point  $s\in \Spec~A$, the strong $\mu$-reduction bundle of $V^s_0$ is $V^s_t$ or $V^s_{t-1}$ and
$$\lim_{p_s\to \infty}\alpha(R_s, I_s) = 1-\mu_{min}(V_t)/d.$$}

Since the notions of $\mu$-reduction and strong $\mu$-reduction  `coincide' in 
characteristic $0$, this says that $\alpha(R, I)$ is always expressed in 
terms of the minimun strong HN slope of the strong $\mu$-reduction bundle.

\vspace{5pt}

We have proved the following result  in [TrW] (Theorem~E) with the additional 
hypothesis that  either $\mbox{Proj}~R$ is nonsingular, or the ideal $I$ is 
   generated by
a set of  homogeneous generators of the same degree (see Theorem~\ref{t3}). 
However though  it is known that $\alpha(R, I) = \alpha(S, IS)$, it is 
not known to us if
$c^I({\bf m}) = c^{IS}({\bf m}S)$. 

\vspace{5pt}
\noindent{\bf Theorem~C}.~~{\it Let $(R, I)$ be a 
 standard graded pair where $R$ is a two dimensional domain. Then  
$$c^I({\bf m}) = \alpha(R, I).$$
In particular $c^I({\bf m}) = c^{IS}({\bf m}S)$, where $S$ denotes
the normalization of $R$ in $Q(R)$.}   

\vspace{5pt}
The following theorem is proved in [TrW] (Theorem~C) when $I$ is 
 an ideal  generated by a set of 
 homogeneous generators of the same degree (see subsection~4.2).

\vspace{5pt}
\noindent{\bf Theorem~D}.~~{\it Let $(R, I)$ be a 
 standard graded pair where $R$ is a two dimensional domain  in 
characteristic $0$ with notations as in Theorem~B,
 then  
\begin{enumerate}
\item $c_{\infty}^I({\bf m}):= \lim_{p_s\to \infty}
c^{I_s}({\bf m}_s)\quad\mbox{exists and}$
\item  For  $p_s\gg 0$,  $c^{I_s}({\bf m}_s) \geq c_{\infty}^I({\bf m})$.
\item If   $V_0$ 
is semistable then
\begin{enumerate}
\item $c_{\infty }^I({\bf m}) = (d_1+\cdots + d_r)/({r-1})$, where 
$M_0 = \oplus_{i=1}^r\sO_X(1-d_i)$
 and
\item for $p_s \gg0$,  
$$c^{I_{s}}({\bf m}_{s}) = c_{\infty }^I({\bf m})
\iff V_0^s~~\mbox{is strongly semistable}.$$
\end{enumerate}
\end{enumerate}}

In particular the $F$-threshold of the {\em reduction mod}~$p_s$, $c^{I_s}({\bf m}_s)$,
 characterizes the  
strong semistability behaviour of the syzygy bundle $V_0$ under 
{\em reduction mod}~$p_s$. 

\vspace{5pt}

Next we analyse the case when $c^{I_{s}}({\bf m}_{s}) \neq 
c_{\infty }^I({\bf m})$.
By Theorem~3.4 and Proposition~3.8 of [HY], where 
 $R=\Z[X_1, \ldots, X_n]$ and $I\subseteq {\bf m} = (X_1, \ldots, X_n)$,
we have
a formula for the log canonical threshold in terms of $F$-pure thresholds (where
$\mbox{fpt}_{\bf m}(I)$ = $c^{\bf m}(I)$ denotes the first jumping 
number of $I$):
$$\mbox{lct}_{\bf m}(I) = \lim_{p\to \infty}
\mbox{fpt}_{{\bf m}_p}(I_p) = \lim_{p\to \infty}c^{{\bf m}_p}(I_p),$$
where ${\bf m}_p$ and $I_p$ are  {\em reductions mod}~$p$ of ${\bf m}$ and $I$,
respectively.

\vspace{5pt}
 K.Schwede asked the following
question: Assuming $\mbox{fpt}_{{\bf m}_p}(f_p)\neq \mbox{lct}_{\bf m}(f)$, is
the
denominator of $\mbox{fpt}_{{\bf m}_p}(f_p)$ (in its reduced form)
 a multiple of $p$?

\vspace{5pt}
In [CHSW] the authors explored the implication of the following two conditions:

\noindent{(1)} the characteristic does not divide the denominator of the $F$-pure threshold.
(2) The $F$-pure threshold and the log canonical threshold coincide.
Theorem~A in [CHSW] and also
the example~4.5 in [MTW] imply that for
 an explicit (nonhomogeneous) polynomial $f$ in a polynomial ring
(note that here  the $F$-pure
threshold $\mbox{fpt}_{{\bf m}_p}(f_p) = c^{{\bf m}_p}(f_p)$), the above two
conditions
could be distinct.

On the other hand, there are examples (see [CHSW] for the references)
of homogeneous polynomials
$f$ of specific types where the two conditions are equivalent.
In [BS] Proposition~5.4, it was shown that for  a
homogeneous
polynomial $f$ of degree $d$ in $R = k[X_0, \ldots, X_n]$
(where $R/(f)$ is an isolated singularity), if $p\geq nd-d-n$ then either
$c^{{\bf m}_p}(f_p) = (n+1)/d$, or  the denominator of $c^{{\bf m}_p}(f_p)$
is a power of $p$. In other words
$$c^{{\bf m}_p}(f_p) \neq  \mbox{lct}_{\bf m}(f) \implies~~
\mbox{the denominator of}~~c^{{\bf m}_p}(f_p)~~\mbox{is a power of}~~p.$$

In this context, here we prove the following (in Section~5).

\vspace{5pt}

\noindent{\bf Theorem~E}.~~{\it Let $(R,I)$ be a standard graded pair, where 
$R$ is a $2$ dimensional domain over an algebraically
closed field $k$ of char $0$.
Let $(R_s, I_s, {\bf m}_s)$ denote  {\em reductions mod $p_s$} of 
$(R, I, {\bf m})$, where  $p_s = \Char~R_s$. Let 
$c^I_{\infty}({\bf m}) = \lim_{p_s\to \infty} c^{I_s}({\bf m}_s)$. Then
for $p_s\gg0$, 
$$ c^{I_s}({\bf m}_s) \neq c^I_{\infty}({\bf m}) \implies
c^{I_s}({\bf m}_s) =   {a_1}/{p_sb_1},$$
 where  $a_1, b_1\in \Z_{+}$ and $\mbox{g.c.d.}(a_1, p_s) = 1$.

In fact, for $p_s\gg 0$,
$$c^{I_s}({\bf m}_s) \neq c^I_{\infty}({\bf m}) \implies
c^{I_s}({\bf m}_s) =  c^I_{\infty}({\bf m}) +\frac{a}{p_s b},$$
for some   $a, b\in \Z_{+}$ such that  
$0< a/b \leq 4(g-1)(r-1)$, where $r+1 =$ the minimal generators of $I$ and
$g =$ the genus of $\mbox{Proj}~R$.}

\vspace{5pt}
However,  there exist examples (Remark~\ref{r3}) where, for all but 
finitely many $p_s$, 
 the  denominators (in its reduced form) of $c^{{\bf m}_s}({\bf m}_s)$
is divisible by $p_s$, but is  not a power of $p_s$.

The organisation of this paper is as follows.

In Section~2, we give a description of the HK density 
function $f_{R,I}$ in terms of the HK density functions of the syzygy 
vector bundles.  Most of the details given here are a rephrasing of the 
details given in  [TrW].

In Section~3 we
introduce  the notion of {\em $\mu$-reduction} and {\em strong $\mu$-reduction} 
bundles, for a choice of 
the sequence of the type (\ref{e22}) (this is a key new idea in the paper).

Then we prove the existence of the $\mu$-reduction and the strong 
$\mu$-reduction bundles,  and 
check  the relevant properties, such as the HN filtration and 
the HK density function of $V_t$ vis-a-vis the HN filtration and the 
HK density function of $V_0$,  
the  relation between the $\mu$-reduction bundle of $V_0$  
and  the 
$\mu$-reduction bundle of $F^{s*}(V_0)$, where $F^s$ is the $s^{th}$-iterated 
 Frobenius map.

In Section~4 
we prove the equality  $c^I({\bf m}) = \alpha(R, I)$ and express this quantity 
 in terms of the minimum strong HN slope of the strong 
$\mu$-reduction bundle of $V_0$. Also in characteristic $0$, we 
realize $c^I_{\infty}({\bf m})$ ($= \alpha^{\infty}(R, I)$) in terms of the 
minimum HN slope of the $\mu$-reduction bundle of $V_0$.

In Section~5, we use the above mentioned characterization  of 
$c^I_{\infty}({\bf m})$ and $c^I({\bf m})$ in terms the invariants   of 
a  vector bundle on 
$\mbox{Proj}~R$, to deduce Theorem~E.

\section{The HK density function in dimension $2$}

Let $X$ be a nonsingular projective curve over an algebraically  
closed field $k$.

We recall the following notations from [TrW]. For details we 
refer the reader to Section~5 of [TrW].

\begin{notations}\label{hnd}Let $V$ be a vector bundle on  $X$. The slope 
of $V$ is $\mu(V) = \deg~V/\rank~V$.
 \begin{enumerate}
\item  The set 
$(\{\mu_1, \mu_2, \cdots, \mu_{t+1}\}, \{r_1, 
\ldots, r_{t+1}\})$ is called
the {\em HN data} of $V$ if
$V$ has the HN filtration  
$$ 0 = F_0 \subset F_1 \subset \cdots \subset 
F_t \subset F_{t+1}= V,$$
with $\mu_i = \mu({F_i}/{F_{i-1}})$
and 
$r_i = \rank({F_i}/{F_{i-1}})$. We call $\mu_i$ 
a {\em HN slope} of $V$ and $r_i$  a {\em HN rank} of $V$. 

We  
denote the minimum HN slope of $V$ by $\mu_{min}(V) = \mu({V}/{F_t})$.

\item If characteristic~$k = p  >0$, then 
$\left(\{a_1, \ldots, a_{l+1}\},
 \{{\tilde r}_1, \ldots, {\tilde r}_{l+1}\}\right)$
is called the {\em strong HN data} of $V$, where 
$m>0$  is an integer
  such that $F^{m*}V$ has the strong HN filtration 
(such an integer $m>0$ exists by Theroem~2.7 of [L])
$$ 0 = E_0 \subset E_1 \subset \cdots \subset 
E_l \subset E_{l+1}= F^{m*}V$$
and 
$a_i = (1/p^m)\mu({E_i}/{E_{i-1}})$ and 
${\tilde r}_i = \rank(E_{i+1}/E_{i}).$. We call $a_i$ a 
 {\em strong HN slope} 
of $V$ and $r_i$  a  strong HN rank of $V$.

 We  denote the minimum strong HN slope 
of $V$ by 
$a_{min}(V) = (1/p^m)\mu({E_{l+1}}/{E_l})$.
 \end{enumerate}
\end{notations}

\begin{rmk}\label{r5}Let $\sO_X(1)$ be an ample line bundle of degree $d$ on 
$X$. Let ${\tilde E}$ be a semistable vector-bundle on $X$ with 
$\mu({\tilde E}) = \mu$ and $\rank({\tilde E}) = r$.
  Then by Serre duality 
$$\begin{array}{lcl}
 m < -{\mu}/{d} & \implies & h^1(X, {\tilde E}(m))
= -r(\mu+dm+(g-1))\\
-{\mu}/{d} \leq m \leq - {\mu}/{d} + (d-3) & \implies &
h^1(X, {\tilde E}(m)) =  C\\
-{\mu}/{d} + (d-3) < m & \implies &
h^1(X, {\tilde E}(m)) =  0,
\end{array}$$
where $|C| \leq r(g-1)$ and $g = \mbox{genus}(X)$.
\end{rmk}

\subsection{The HK density functions for vector bundles on curves}

Let $X$ be a nonsingular projective curve  over an 
algebraically closed  field of characteristic $p>0$. Let $\sO_X(1)$ be an
ample line bundle of degree $d$ on $X$.
Let $V$ be a  vector bundle on $X$. 

We recall the definition ((6.1) in [TrW]) of  the HK density 
function of $V$ with respect to $\sO_X(1)$.
 Let $f_n(V,\sO_X(1)):\R\longto [0, \infty)$ be given by (where $q=p^n$) 
$$f_n(V,\sO_X(1))(x) = \frac{1}{q}h^1(X, F^{n*}V(\lfloor (x-1)q\rfloor)).$$
and let
\begin{equation}\label{fv}
f_{V, \sO_X(1)}:\R\longto [0, \infty)~~~\mbox{given by}~~~ 
x \to  \lim_{n\to \infty}f_n(V,\sO_X(1))(x)
\end{equation} 
The function $f_{V,\sO_X(1)}$ is 
 well defined and continuous (though need not be compactly supported).

\begin{rmk}\label{*1}Later in the paper, we will use the following 
formula (given in terms 
of the strong HN data 
$(\{a_1, \ldots, a_{l+1}\}, \{r_1, \ldots, r_{l+1}\})$ of $V$.

We choose 
  $n_1>0$ such that 
$F^{n_1*}V$ has the strong HN filtration
$$0 = E_{0} \subset E_{1} \subset \cdots \subset E_{l} \subset E_{l+1} =
F^{n_1*}V, $$
where $a_i = (1/p^{n_1})\mu(E_i/E_{i-1})$ and 
$r_i  = \rank(E_i/E_{i-1})$.

Since $a_1 > a_2 > \cdots > a_{l+1}$, we can choose $q \gg 0$ 
($q_1 = p^{n_1}$) such that 
$$-\frac{ a_1qq_1}{d} < 
-\frac{ a_1qq_1}{d}+(d-3) < -\frac{a_2qq_1}{d} < 
-\frac{ a_2qq_1}{d}+(d-3) <\cdots < 
-\frac{ a_{l+1}qq_1}{d}.$$

\vspace{5pt}

\noindent{(1)}\quad By Remark~\ref{r5} (where $q=p^n$)
$$qq_1f_{n+n_1}(V,\sO_X(1))(\frac{m}{qq_1}) = h^1(X, F^{n+n_1*}V (m-qq_1)) = 
\sum_{i=1}^{l+1}h^1(X, F^{n*}(E_i/E_{i-1})(m-qq_1)).$$

If $g = \mbox{genus} (X)$ and $R_i = r_i\left[a_i+d(\frac{m}{qq_1}-1)+\frac{(g-1)}{qq_1}\right]$ then we have 
\vspace{5pt}

$$f_{n+n_1}(V,\sO_X(1))(\frac{m}{qq_1}) = \begin{cases} 
-\sum_{i=1}^{l+1}R_i & \mbox{for}\quad \frac{m}{qq_1} < 
1-\frac{a_1}{d}\\\\
 \frac{C_1}{qq_1}-\sum_{i=2}^{l+1}R_i & \mbox{for}\quad
1-\frac{a_1}{d} \leq \frac{m}{qq_1} < 1 -\frac{a_1}{d}+\frac{(d-3)}{qq_1}\\\\
-\sum_{k=i+1}^{l+1}R_k & \mbox{for}\quad  1-\frac{a_i}{d}+\frac{(d-3)}{qq_1}
\leq  \frac{m}{qq_1} < 1-\frac{a_{i+1}}{d}\\\\
\frac{C_{i+1}}{qq_1}-\sum_{k=i+2}^{l+1}R_k & \mbox{for}\quad  
1-\frac{a_{i+1}}{d}\leq  \frac{m}{qq_1} 
\leq 1-\frac{a_{i+1}}{d} +\frac{(d-3)}{qq_1}\\\\
 0 & \mbox{for}\quad 1 -\frac{a_{l+1}}{d} +\frac{(d-3)}{qq_1} \leq 
\frac{m}{qq_1}, 
\end{cases}$$
where  $|C_i|\leq \rank(V)(g-1)$ for all $i$ and $a_{l+1} = 
 a_{min}(V)$.

\vspace{5pt}

\noindent{(2)}\quad Taking limit as $n\to \infty$, we get the formula for $f_{V, \sO_X(1)}$:
$$f_{V,\sO_X(1)}(x) = 
\begin{cases}
-\left[\sum_{i=1}^{l+1}a_ir_i+d(x-1)r_i\right]
& \quad\mbox{for}\quad x < 1-a_1/d\\\ 
- \left[\sum_{k={i+1}}^{l+1}a_kr_k+d(x-1)r_k\right] 
& \quad\mbox{for}\quad 1-a_i/d  \leq x < 1-a_{i+1}/d.\end{cases}$$

\vspace{5pt}

\noindent{(3)}\quad   
$\mbox{Support}~f_{V,\sO_X(1)} \subseteq \mbox{the interval}~(-\infty, 1-a_{min}(V)/d]$
and 
\begin{equation}\label{*3}\alpha(V, \sO_X(1)) := \mbox{Sup}~\{x\mid 
f_{V, \sO_X(1)}(x) >0\}
 = 1-\frac{a_{min}(V)}{d}.\end{equation}
\end{rmk}

\begin{rmk}Replacing $R$ by $R\tensor_k{\bar k}$ does not change the function
$f_{R,I}$ and the semistability behaviour of any vector bundle $V$ on 
$X= \mbox{Proj}~R$. Therefore   we can assume, without loss of generality, that the 
underlying field $k$ is algebraically closed.\end{rmk}

\subsection{The HK density functions of $f_{R,I}$ and the syzygy vector bundles} 

 Let $(R,I)$ be  a  standard graded pair, where $R$ is a domain defined over 
a field of characteristic $p>0$.

Let $S = \oplus_m S_m$ be the integral closure of $R$ in its quotient field.
Then the  inclusion map $\pi:R\longrightarrow S$  is 
a graded finite map of degree $0$, where $S$ is a normal domain and $Q(R) = Q(S)$.
The additivity of the HK density function (Proposition~2.14 of [T2]) 
implies that 
$$f_{R, I}(x) = f_{S, I}(x) = \lim_{n\to \infty} f_n(x) = 
\lim_{n\to \infty}\frac{1}{q} \ell\left(\frac{S}{I^{[q]}S}\right)_{\lfloor 
xq\rfloor}.$$
Since $R$ is a standard graded ring over $k$, the canonical embedding 
 $Y = {\rm Proj}~R \longto \P_k^n$ 
gives the very ample line bundle $\sO_Y(1)$ on  $Y$.
Let  $X = {\rm Proj}~S$  with the canonical map  $\pi:X\longto Y$ and 
let $\sO_X(1) =  \pi^*\sO_Y(1)$ be the ample line bundle on $X$.

Note that $X$  is a nonsingular projective curve.   
 For a choice of homogeneous generators $h_1, \ldots, h_\mu$ of $I$ 
of degrees  $d_1, \ldots, d_{\mu}$, we have
the canonical (locally split) exact
 sequence of locally free sheaves of $\sO_X$-modules
\begin{equation}\label{e2}
0\longto V_0 \longto M_0 = \oplus_{i=1}^{\mu}\sO_X(1-d_i)\longto \sO_X(1)
\longto 0,\end{equation}
where the map 
$\sO_X(1-d_i)\longto \sO_X(1)$ is given by the multiplication by the
 element  $h_i$.

Since, for $q = p^n \gg 0$, 
 $$f_n\left(\frac{m+q}{q}\right) = \frac{1}{q} 
\ell\left(\frac{S}{I^{[q]}S}\right)_{m+q} = 
 \frac{1}{q}\left[h^1(X, (F^{n*}V_0)(m)) -
 h^1(X, (F^{n*}M_0)(m))\right]$$
we have
\begin{equation}\label{fr}f_{R, I}(x) = 
f_{V_0, \sO_X(1)}(x)-f_{M_0, \sO_X(1)}(x), ~~\mbox{for}~~x \geq 1.\end{equation}

If $a_{min}(V_0) < a_{min}(M_0)$ then by (\ref{*3})
$\alpha(R,I) = 1-a_{min}(V_0)/d$.
This  holds true when 
 $d_1 = \cdots = d_s$,  as  $M_0$ is strongly semistable and therefore
$\mu(M_0) = a_{min}(M_0)$ and $a_{min}(V_0) \leq \mu(V) < \mu(M_0)$. 

However, it may happen that  $a_{min}(V_0) = a_{min}(M_0)$ and the 
HK density functions for $V_0$ and $M_0$ may coincide in a 
neighbourhood of their common maximum support point (see   
Remark~\ref{rrr}).

In the next section we introduce  the 
notion of $\mu$-{\em reduction} and the strong $\mu$-{\em reduction} for a short 
exact sequence of type (\ref{e2}). 
Using the strong $\mu$-reduction bundle $V_{t_0}$ ($\subset  V_0$) we replace the 
short exact
sequence (\ref{e2}) by 
another sequence
$$0\longto V_{t_0}\longto M_{t_0}\longto \sO_X(1)\longto 0$$ such that 
\begin{enumerate}
\item  $a_{min}(V_{t_0}) < a_{min}(M_{t_0})$ and 
\item  $f_{V_{0}, \sO_X(1)}-f_{M_{0}, \sO_X(1)} = 
f_{V_{t_0}, \sO_X(1)}-f_{M_{t_0}, \sO_X(1)}$ and 
\item  $a_{min}(V_{t_0})$ occurs in the strong HN data of $V_0$.
\end{enumerate}

In particular, we  express 
$\alpha(R, I)$   in terms of one of the  strong HN slopes of 
the syzygy  vector bundle $V_0$. 
In characteristic~$0$, using the $\mu$-reduction bundle $V_t$ (whose 
minimum HN slope occurs in the HN data of $V_0$) we are able to 
express $\alpha^{\infty}(R, I)$ (the maximum support point of the 
HK density function in characteristic $0$)
in terms of one of the HN slopes of $V_0$.

Using this formula for $\alpha(R, I)$, in terms of the strong HN data of 
a single vector bundle, and  Remark~\ref{*1}~(1), 
we are able to  prove the  equality 
$c^I({\bf m}) = \alpha(R, I)$.
This enables us to  study 
 various properties (Theorem~D and  Theorem~E)  of  
the $F$-thresholds $c^I({\bf m})$  and $c^I_{\infty}({\bf m})$ for two 
dimensional standard graded pair $(R, I)$.

\section{$\mu$-reduction and strong $\mu$-reduction bundles}

Let $X$ denote a nonsingular projective curve with an ample line 
bundle $\sO_X(1)$ of  degree $d$ over a 
field $k$ of arbitrary characteristic and let
\begin{equation}\label{*p}0\longto V_0\longby{f_0} 
M_0 = \oplus_{i=1}^\mu\sO_X(1-d_i) \longto \sL = \sO_X(1)\longto 0
\end{equation} be a 
short exact sequence of  sheaves of $\sO_X$-modules, where 
$d_1\leq d_2\cdots \leq d_\mu$ are positive intergers and 
where 
the map $\sO_X(1-d_i) \longto \sO_X(1)$ is 
a multiplication map given by $h_i\in H^0(X, \sO_X(d_i)$.

\begin{notations}\label{mu}
For the sequence~(\ref{*p}), we denote the  
HN filtration of $M_0$ by
$$0 \subset M_{{l_1}-1} \subset  \cdots \subset M_0 \quad\mbox{and let}$$
$$ V_{l_1-1} \subseteq \cdots 
\subseteq V_1 \subseteq V_0$$
denote the induced (need not be the HN) filtration on $V_0$, where 
$V_i = M_i\cap V_0$.
 For every 
$0\leq i\leq l_1-1$, let 
$f_i: V_i\longto  M_i$ be the 
canonical inclusion map.

Here $M_i$ can explicitly be given as follows:
Let  $\{d_1, \ldots, d_\mu\} = \{{\tilde d_1},\ldots, {\tilde d_{l_1}}\}$, 
with ${\tilde d_{l_1}} > {\tilde d_{l_1-1}} >\cdots >{\tilde d_1}$.
Then  $M_{l_1-i} = \oplus \sO_X(1-{\tilde d_1})\oplus \cdots \oplus 
\sO_X(1-{\tilde d_i})$.
In particular  $\mu_{max}(M_0) = (1- {\tilde d_1})d$ and 
$\mu_{min}(M_0) = (1-{\tilde d}_{l_1})d$, where $d$ is the degree of $\sL$.

It is easy to check that 
 the bundle $V_{l_1-1} = 0$ iff $M_{l_1-1}$ is a  line bundle.

\end{notations}

\subsection{The $\mu$-reduction bundle}

\vspace{5pt}

\begin{defn}\label{mur}
The bundle $V_t$ is the {\em $\mu$-reduction bundle} of $V_0$ 
(of sequence~(\ref{*p}) if $t < l_1$ such that $V_t \neq 0$ and 
\begin{enumerate}
\item $\mu_{min}(V_i) = \mu_{min}(M_i)$ for  $i < t$
  and 
\item $\mu_{min}(V_t) <\mu_{min}(M_t)$.\end{enumerate}

Strictly speaking we should be refering to $V_t$ as the $\mu$-reduction bundle of 
the sequence (\ref{*p}) as the notion depends on the  sequence
(\ref{*p}) too. Since  in the paper there would not be any
 ambiguity about the associated 
sequence, we will refer the bundle $V_t$ as the $\mu$-reduction 
bundle of $V_0$.
\end{defn}

\vspace{5pt}

Next we prove relevant properties of the filtration~$\{V_i\}_i$ 
and then  prove  the existence of the $\mu$-reduction bundle of $V_0$.

We would repeatedly use the following two obvious properties of the  
sequence~(\ref{*p}).
(1)  The induced map $M_{l_1-1}\longto \sL$ is nonzero 
and (2) $\mu_{max}(M_0) = \mu(M_{l_1-1}) < \mu(\sL)$.

\begin{rmk}\label{r4} The following are well known facts (can also 
refer to Remark~5.5 in [TrW]).
\begin{enumerate}
 \item If $0\longto V'\longto V\longto V''\longto 0$ is a short exact
sequence of nonzero
vector bundles on $X$, then
\item either $\mu(V') \leq \mu(V) \leq \mu(V'')$
or $\mu(V') \geq \mu(V) \geq \mu(V'')$.
\item For a nonzero  map of bundles  $E\longto W$,
 where $W$ is semistable,
 $\mu_{min}(E) \leq \mu(W)$.
In particular,  if $0\longto V'\longto V\longto V/V'\longto 0$ is an exact 
sequence  of nonzero bundles such that $V/V'$ is semistable and  
 $W\subseteq V$ is a nonzero bundle such that $\mu_{min}(W) > 
\mu(V/V')$ then  $W\subseteq V'$.
\item For a nonzero bundle $V$ on $X$, we have $\mu_{min}(F^{m*}(V)) \leq p^m\mu_{min}(V)$, for any $m\geq 1$.
\end{enumerate}
\end{rmk}

\begin{lemma}\label{lmu} \begin{enumerate}
\item The sequence 
$V_{l_1-1} \subset \cdots \subset V_1 \subset V_0$
is a sequence of distinct subbundles and
\item $\mu_{min}(V_j) \leq \mu_{min}(M_j)$, for $0\leq j < l_1-1$ and 
same holds for $j=l_1-1$ if the bundle $V_{l_1-1}$ is nonzero.
\item If $i < l_1-1$ such that $\mu_{min}(V_j) = \mu_{min}(M_j)$, for
$0\leq j\leq i$, then the canonical sequence 
 $$0\longto V_{i+1} \longby{f_{i+1}} M_{i+1}\longto \sL\longto 0$$
is a short exact sequence
and  $V_j/V_{j+1} \simeq M_j/M_{j+1}$, for 
$0\leq j\leq i$.
\end{enumerate}
\end{lemma}
\begin{proof}For $0\leq i \leq l_1-1$,
the induced map $M_i\longto \sL$ is nonzero and  factors through 
the injective map $M_i/f_i(V_i)\longto \sL$.
This implies $\coker~f_i \neq 0$,
 for every  $0\leq i \leq l_1-1$.

\vspace{5pt}
\noindent{(1)}\quad 
If $V_i = V_{i+1}$, for some $i < l_1-1$  then we have $M_i/M_{i+1} \simeq 
\coker~f_i/\coker~f_{i+1}$, where 
$\coker~f_i/\coker~f_{i+1}$ is a subquotient (but not a subsheaf) of $\sL$, and hence a torsion-sheaf of 
$\sO_X$-modules, on the other hand $M_i/M_{i+1}$ 
is a nonzero locally free sheaf. Hence $\coker~f_i/\coker~f_{i+1} = 0$.
\vspace{5pt}

\noindent{(2)}\quad This follows as 
$0\longto V_i/V_{i+1}\longto M_i/M_{i+1}$ implies 
$$\mu_{min}(V_i) \leq \mu(V_i/V_{i+1})\leq \mu(M_i/M_{i+1}) 
= \mu_{min}(M_i).$$ 

\noindent{(3)}\quad Note that $\coker f_0 = \sL$. 
 It is enough to prove that if there is $l_1-1 > j\geq 0$ such 
that $\mu_{min}(V_j) = \mu_{min}(M_j)$ and $\coker f_j = \sL$ then 
$\coker~f_{j+1} = \sL$ and $V_j/V_{j+1} \simeq M_j/M_{j+1}$.
Consider the short exact sequence 
$$0\longto V_j/V_{j+1}\longto M_j/M_{j+1}\longto \sL/\coker f_{j+1}\longto 0,$$
Now $\sL/\coker f_{j+1}$ is a torsion sheaf. Also 
$\mu(V_j/V_{j+1}) = \mu(M_j/M_{j+1})$ (as argued in (2)).
Therefore 
$$\rank~\frac{V_j}{V_{j+1}} = \rank~\frac{M_j}{M_{j+1}}\implies  
\deg~\frac{V_j}{V_{j+1}} = \deg~\frac{M_j}{M_{j+1}}.$$
 Hence $\deg~(\sL/\coker~f_{j+1})= 
\ell(\sL/\coker~f_{j+1}) = 0$ which implies
$\coker~f_{j+1} = \sL$ and 
hence $V_j/V_{j+1} \simeq M_j/M_{j+1}$.
\end{proof}

\begin{propose}\label{L0} The bundle $V_0$ has 
 $\mu$-reduction bundle $V_t$, for some $t<l_1$.
\end{propose}
\begin{proof}If the bundle $V_{l_1-1} = 0$ then $V_0$
has $\mu$-reduction bundle for some $t < l_1-1$, otherwise, 
by Lemma~\ref{lmu}~(3), we have  $M_{l_1-1} \simeq \sL$.

Hence we can assume that $V_{l_1-1}\neq 0$.

 Suppose $\mu_{min}(V_i) = \mu_{min}(M_i)$, for every $0\leq i\leq l_1-1$. 
Then, by Lemma~\ref{lmu}~(3), 
the sequence  $0\longto V_{l_1-1}\longto M_{l_1-1} \longto \sL\longto 0$
is exact.
Now, as $M_{l_1-1}$ is semistable, we have 
$$\mu_{min}(V_{l_1-1}) \leq \mu(V_{l_1-1}) \leq \mu(M_{l_1-1}) = 
\mu_{min}(M_{l_1-1}).$$
But then we have the equality
 $\mu(V_{l_1-1}) = \mu(M_{l_1-1}) = \mu(\sL)$.
Hence there is $t'<l_1$ such that $\mu_{min}(V_{t'}) <\mu_{min}(M_{t'})$. 
The smallest number $t<l_1$ such that $\mu_{min}(V_t) < \mu_{min}(M_t)$
gives the $\mu$-reduction bundle $V_t$ of $V_0$.
\end{proof}

Though the bundle $V_t$ may not  occur in the HN filtration of $V_0$, we can 
relate the HN filtration of $V_t$ and the HN filtration of $V_0$.

\begin{lemma}\label{l01} Let $V_t$ be the $\mu$-reduction bundle 
of $V_0$, where $t\geq 1$. Then the 
HN filtration of $V_0$ is
 $$\cdots\subset  W_{l+1}\subset W_l\subset V_{t-1} 
\subset V_{t-2}\cdots \subset V_1\subset  V_0.$$
Moreover
\begin{enumerate}
\item 
$W_{l}\subseteq V_t \subset V_{t-1}$ and
\item  
the HN filtration of $V_t$ is 
\begin{enumerate}
\item[(a)] either  $\cdots\subset  W_{l+1}\subset  W_l = V_t$ 
(equivalently $\mu_{min}(V_t) > \mu_{min}(V_{t-1}))$, 
\item[(b)]  or
$\cdots\subset  W_{l+1}\subset W_l \subset V_t$, 
 (equivalently $\mu_{min}(V_t) = 
\mu_{min}(V_{t-1}))$.
\end{enumerate}
In both the cases
$\mu_{min}(V_{t-1}) = \mu({V_{t-1}}/{W_l}) = 
\mu({V_{t-1}}/{V_t})$.
\end{enumerate}
\end{lemma}
\begin{proof} By Lemma~\ref{lmu}~(3), we have 
$V_i/V_{i+1} \simeq M_i/M_{i+1}$, for all $0\leq i<t$.
Let 
the HN filtration of $V_{t-1}$ be 
$\cdots \subset W_{l+1}\subset W_l\subset V_{t-1}$.
Then $$\mu\left(\frac{V_{t-1}}{W_l}\right) = \mu_{min}(V_{t-1}) = 
\mu_{min}(M_{t-1})
=\mu\left(\frac{V_{t-1}}{V_t}\right) >\mu\left(\frac{V_{t-2}}{V_{t-1}}\right) >
\cdots >\mu\left(\frac{V_0}{V_1}\right),$$
Hence, by the uniqueness property of the HN filtration, 
  the HN filtration of $V_0$ has to be the filtration
$$\cdots \subset W_{l+1}\subset W_l\subset V_{t-1}\subset V_{t-2}\subset 
\cdots \subset V_1\subset V_0$$
and  $V_{t-1}/V_t$ is semistable.
Moreover, by Remark~\ref{r4}, the inquality  $\mu_{min}(W_l) > 
\mu(V_{t-1}/V_t)$ implies $W_l \subseteq V_t$.

If  $W_l = V_t$  then $$\mu_{min}(V_t) = \mu(W_l/W_{l+1}) > 
\mu(V_{t-1}/W_l) = 
\mu_{min}(V_{t-1})$$ 
and the HN filtration for $V_t$
is $\cdots\subset W_{l+1} \subset W_l = V_t$.

If $W_l\subset V_t$ then
the exact sequence 
$$0\longto V_t/W_l\longto V_{t-1}/W_l\longto V_{t-1}/V_t\longto 0$$ implies
$\mu(V_{t}/W_l) =\mu(V_{t-1}/W_l)$ and $V_t/W_l$ is semistable. 
Hence the HN filtration for $V_t$ is 
$\cdots\subset W_{l+1} \subset W_l \subset V_t$.
\end{proof}

\begin{rmk}\label{ree}If $V_t$ is the $\mu$-reduction bundle of $V_0$ such 
that $t\geq 1$ then,  by Lemma~\ref{l01}~(2), we have
 $\mu_{min}(M_{t-1})\leq  \mu_{min}(V_{t}) < \mu_{min}(M_{t})$.
\end{rmk}

\subsection{The strong $\mu$-reduction bundle}
Let $X$ be  a nonsingular curve over an algebraically closed field  $k$ of
$\Char~p>0$. Let $$0\longto V_0\longto M_0\longto \sL \longto 0$$ be the
 sequence (\ref{*p}).  Since this is an exact  sequence of locally 
free sheaves, for  any $s>0$ if $F^s:X\longto X$ is 
the $s^{th}$-iterated Frobenius map  then
the induced map  (here $q=p^s$)

\begin{equation}\label{fe2}0\longto F^{s*}V_0\longto F^{s*}M_0 = 
\oplus_{i=1}^\mu\sO_X(q-qd_i)\longto F^{s*}\sL = \sO_X(q) \longto 0
\end{equation} is exact.

\begin{rmk}\label{r00} The filtration
$$0 \subset F^{s*}M_{{l_1}-1} \subset  \cdots \subset F^{s*}M_0 = 
F^{s*}M$$
is the HN filtration of $F^{s*}M_0$.
Moreover,  $X$ being nonsingular implies that the map
$F^s$ is flat and therefore $F^{s*}M_i \cap F^{s*}V_0 = 
F^{s*}V_i$.
In particular
the induced filtration on $F^{s*}V_0$ is 
$$F^{s*}V_{l_1-1} \subset \cdots \subset 
F^{s*}V_1 \subset F^{s*}V_0.$$
\end{rmk}

\vspace{5pt}

\begin{defn}\label{murp} The bundle $V_{t_0}$ is the 
{\em  strong $\mu$-reduction} bundle of $V_0$ 
if the bundle $F^{m_1*}V_{t_0}$ is the $\mu$-reduction bundle of 
$F^{m_1*}V_0$, where
  $m_1\geq 0$ is an integer such that  
the HN filtration of $F^{m_1*}V_0$ is the strong HN filtration (this exists 
by [L]).

By Lemma~(\ref{L0}), the strong $\mu$-reduction bundle $V_{t_0}$ does exist.
 and $t_0 <l_1$ is  the integer such that $a_{min}(V_{t_0}) <  a_{min}(M_0)$ and
$a_{min}(V_i) =  a_{min}(M_i)$, for every $0\leq i < t_0$.

\end{defn}

\begin{rmk}\label{e*}All the succeeding results of this section hold true 
(with exactly the same proofs) for any short exact sequence of locally free sheaves of $\sO_X$-modules
$$0\longto V_0\longto M_0\longto \sL\longto 0,$$
where $\sL$ is a line bundle, satisfying the following properties (P1) and (P2),
\begin{enumerate} 
\item[(P1)]  The induced map $M_{l_1-1}\longto \sL$ is nonzero 
and $\mu_{max}(M_0) < \mu(\sL)$, where $M_{l_1-1}$ is the 
first  nonzero bundle occuring in the HN filtration of $M_0$.
\item[(P2)] If $\Char~k = p>0$ then the HN filtration of $M_0$ is the strong 
HN filtration.
\end{enumerate}
\end{rmk}

The following lemma implies that the 
strong $\mu$-reduction bundle always contains the $\mu$-reduction bundle. 

\begin{lemma}\label{L0p} For $s\geq 1$, if $F^{s*}{V_{t_1}}$ is the 
$\mu$-reduction bundle of $F^{s*}V_0$ and $V_t$ is 
the $\mu$-reduction bundle of $V_0$ then
$t_1\leq t$.

In particular if $V_{t_0}$ is the strong $\mu$-reduction bundle of $V_0$ then
$t_0\leq t$.
\end{lemma} 
\begin{proof}We know  $t_1<l_1$.
 By Remark~\ref{r00}, 
$ F^{s*}M_i/F^{s*}M_{i+1} \simeq F^{s*}(M_i/M_{i+1})$ and 
$F^{s*}M_i\cap F^{s*}V_0 = F^{s*}V_{i}$.
By definition,  $\mu_{min}(V_t) <\mu_{min}(M_t)$ therefore (see Remark~\ref{r4})
$$\mu_{min}(F^{s*}V_t) \leq p^s\mu_{min}(V_t) < p^s\mu_{min}(M_t) =
\mu_{min}(F^{s*}M_t),$$
which implies  $t_1\leq t$.
\end{proof}

\begin{rmk}\label{rmu}
Though $V_t$ may not occur in the HN filtration of $V_0$,  
the number $\mu_{min}(V_t)$ is equal to one of the HN slopes of $V_0$, by 
Lemma~\ref{l01}.
Similarly, if $V_{t_0}$ is 
the  strong $\mu$-reduction bundle of $V_0$  then the number 
$a_{min}(V_{t_0})$  is equal to one of the strong HN slopes of $V_0$.
\end{rmk}

\begin{lemma}\label{hk2}Let $X$ be a nonsingular projective 
curve over a field of char~$p>0$ with an ample
line bundle $\sO_X(1)$ of degree $d$. Let (where $d_i\geq 1$) 
$$0\longto V_0\longto M_0 = \oplus_i\sO_X(1-d_i)\longto \sO_X(1)\longto 0,$$
 be a short exact sequence of locally free sheaves of $\sO_X$-modules. 
If  $V_{t_0}$ is  
 the strong $\mu$-reduction bundle  of $V_0$ then 

\begin{enumerate}
\item $ f_{V_0, \sO_{X}(1)}- f_{M_0, \sO_{X}(1)} = 
f_{V_{t_0}, \sO_{X}(1)}- f_{M_{t_0}, \sO_{X}(1)}$, and
\item $\mbox{max}~\{x\mid f_{V_0, \sO_{X}(1)}(x)- f_{M_0, 
\sO_{X}(1)}(x)\neq 0\} =
1-\frac{a_{min}(V_{t_0})}{d}$.
\end{enumerate}
\end{lemma}
\begin{proof}If $t_0 = 0$ then the assertion~(2) follows from 
(\ref{*3}) and the assertion~(1). Hence 
we can assume $t_0\geq 1$. 
Let $n_1 > 0$ such that the HN filtration of 
$F^{n_1*}V_0$ is the  strong HN filtration.
If $V_{t_0}$ is the strong $\mu$-reduction bundle of $V_0$ then, by definition,  
$F^{n_1*}V_{t_0}$ is the $\mu$-reduction bundle of $F^{n_1*}V_0$. Hence

 \begin{enumerate}
\item $\mu_{min}(F^{n_1*}V_{t_0}) 
< \mu_{min}(F^{n_1*}M_{t_0})$ and,  by  Lemma~\ref{l01},
  \item  the HN filtration  of $F^{n_1*} V_{0}$ is
$$0\subset \cdots \subset {\tilde W}_{l+1}\subset {\tilde W}_l \subset 
F^{n_1*} V_{{t_0-1}} \subset F^{n_1*} V_{t_0-2}\subset
\cdots  \subset F^{n_1*}V_0$$
and  
\item 
\begin{enumerate}
\item[(a)] either the HN filtration of ${F^{n_1*} V}_{t_0}$ is
$\cdots \subset {\tilde W}_{l+1}\subset {\tilde W}_l = {F^{n_1*} V}_{t_0}$ 
\item[(b)] or the HN filtration of ${F^{n_1*} V}_{t_0}$ is  $\cdots 
\subset {\tilde W}_{l+1}\subset {\tilde W}_l \subset 
 {F^{n_1*} V}_{t_0}$.
\end{enumerate}
\item Moreover, in both the cases,
 $$\frac{F^{n_1*}V_0}{F^{n_1*}V_1}\simeq 
\frac{F^{n_1*} M_0}{{F^{n_1*} M}_1}, \ldots, 
\frac{{F^{n_1*} V}_{t_0-1}}{{F^{n_1*} V}_{t_0}}\simeq 
\frac{{F^{n_1*} M}_{t_0-1}}{{F^{n_1*} M}_{t_0}}$$
and $\mu({{F^{n_1*} V}_{t_0-1}}/{{F^{n_1*} V}_{t_0}}) = 
\mu({{F^{n_1*} V}_{t_0-1}}/{{\tilde W}_l})$.
\end{enumerate}

It is easy to check that 
 the HN filtration of ${F^{n_1*} V}_{t_0}$ is the strong HN filtration.

Moreover, if
$(\{a_1q_1, \ldots, a_{k+1}q_1\}, \{r_1, \ldots, r_{k+1}\})$
is the strong HN data of $F^{n_1*}V_{t_0}$ then
$(\{a_1, \ldots, a_{k+1}\}, \{r_1, \ldots, r_{k+1}\})$
is the strong HN data of $V_{t_0}$.
Let the  HN data (which is same as the strong HN data) for $M_{t_0}$ be 
$(\{b_1, \ldots, b_{l_1-t_0}\}, \{s_1, \ldots, r_{l_1-t_0}\})$. 

 Let
$$A_n(m) = 
h^1(X, F^{n+n_1*}{V}_{0}(m)) - h^1(X, 
F^{n+n_1*}{M}_0(m)),$$
$$B_n(m) = 
h^1(X, F^{n+n_1*}{V}_{t_0}(m)) -h^1(X, 
F^{n+n_1*}{M}_{t_0}(m)).$$

\noindent{\bf Claim}. 
\begin{enumerate}\item For $q=p^n$ and there is a constant $C$ such that 
 $|C|\leq (\rank~M_0)d(d-3)$ and 
$$\begin{array}{ll}
A_{n}(m) &  =
 B_{n}(m) + C,\quad\mbox{for}\quad \frac{m}{qq_1} \in~~\left[0, 
\frac{(d-3)}{qq_1} -\frac{a_{k+1}}{d}\right),\\\\
A_{n}(m) & = B_{n}(m) = 0,\quad
\mbox{for}\quad \frac{m}{qq_1} \in~~~\left[\frac{(d-3)}{qq_1} -
\frac{a_{k+1}}{d}, \infty\right).\end{array}$$
\item $A_n(m) = B_n(m) = h^1(X, F^{n+n_1*}{V_{t_0}}(m)) = 
-r_{k+1}[a_{k+1}qq_1 + md +d(d-3)],$

$$\mbox{for}~~~~\frac{m}{qq_1} \in \left(\frac{(d-3)}{qq_1}- 
\frac{ \mbox{min}\{a_k, 
b_{l_1-t_0}\}}{d},~~~
-\frac{a_{k+1}}{d}\right).$$
\end{enumerate}

\vspace{5pt}
\noindent{\underline{Proof of the claim}:}\quad 
We prove the claim when ${\tilde W}_l\subset F^{n_1*}V_{t_0}$. The case 
${\tilde W}_l =  {F^{n_1*}V}_{t_0}$
 can be argued similarly.
Since $$a_{k+1}q_1 = \mu(F^{n_1*}V_{t_0}/{\tilde W}_l) = 
\mu(F^{n_1*}V_{t_0-1}/{\tilde W}_l)
= \mu(F^{n_1*}(M_{t_0-1}/M_{t_0}))$$
the  strong HN data of $V_0$ is 
$(\{a_1, \ldots, a_{k+1}, a_{k+2}, \ldots, a_{k+t_0}\}, \{r_1, \ldots, r_k, 
{\overline{r_{k+1}}}, \ldots, {\overline{r_{k+t_0}}}\})$.
Hence the strong HN data of $M_0$ is given by
$$ (\{b_1, \ldots, b_{l_1-t_0}, a_{k+1}, a_{k+2}, \ldots, a_{k+t_0}\}, 
\{s_1, \ldots, s_{l_1-t_0}, 
{\overline{r_{k+1}}}-r_{k+1}, {\overline{r_{k+2}}}\ldots, {\overline{r_{k+t_0}}}\})$$
as 
$$s_{l_1-(t_0-1)} = \rank (M_{t_0-1}/M_{t_0}) = \rank (V_{t_0-1}/V_{t_0})
= {\overline{r_{k+1}}}-r_{k+1}.$$
 Now the claim follows from the formula given in Remark~\ref{*1}~(1).

Therefore we have
$$\lim_{q\to \infty}\frac{1}{qq_1}A_n(\lfloor xqq_1\rfloor) = 
\lim_{q\to \infty}\frac{1}{qq_1}B_n(\lfloor xqq_1\rfloor).$$
This proves  assertion~(1) of the lemma.
The part~(1) of the claim also implies that 
$$f_{V_0, \sO_{X}(1)}(x)- f_{M_0, \sO_{X}(1)}(x)
   =  0,\quad\mbox{for}\quad x\in\left[ 1- \frac{a_{k+1}}{d}, 
\quad \infty\right).$$ 
Note that $a_{k+1} < a_k$ and 
$a_{k+1} = a_{min}(V_{t_0})  < 
b_{l_1-t_0} = a_{min}(M_{t_0})$.

Hence if
$x\in \left(1-\min\{{a_k}/{d}, {a_{min}(M_{t_0})}/{d}\},\quad 1- 
{a_{min}(V_{t_0}}/{d}\right)$ then 
$$f_{V_0, \sO_{X}(1)}(x)- f_{M_0, \sO_{X}(1)}(x)
 =  -r_{k+1}\left[{a_{k+1}} +d(x-1)\right] >0.$$  
This proves the second assertion and hence the lemma. 
\end{proof}

\section{The maximum support 
 $\alpha(R, I)$ and the  $F$-threshold $c^I({\bf m})$}

Throughout this section fix the following

\begin{notations}\label{n4}Let $(R, I)$ be a standard graded pair, where $R$ is a 
two dimensional domain over an algebraically closed field $k$. 
Let $d = e_0(R, {\bf m})$ be the multiplicity of $R$ with respect to ${\bf m}$.
 In the rest of this section we fix a set of homogeneous generators
$h_1,\ldots, h_\mu$ 
 of degress
$d_1, \ldots, d_\mu$ respectively, of $I$.
Let $S$ be the integral closure of $R$ in its quotient field.  
Then $X= \mbox{Proj}~S$ is a  
nonsingular  curve with the ample line bundle $\sO_X(1)$ of
degree $d$ and the short exact sequence 
\begin{equation}\label{ch0}0\longto V_0\longto M_0 =\oplus_{i=1}^\mu
\sO_{X}(1-d_i)\longto \sO_{X}(1)
\longto 0,\end{equation}
where the map $\sO_X(1-d_i)\longto \sO_X(1)$ is the multiplication map
given by the element $h_i$.
 
Let the HN filtration of $M$ be
$$0 = M_{l_1} \subset M_{l_1-1}\subset\cdots M_1\subset 
M_0 = M,~~\mbox{and let}~~V_i = V\cap M_i.$$
\end{notations}

By Proposition~\ref{L0},   the bundle $V_0$
has the $\mu$-reduction bundle $V_t$ for some $t<l_1$
  and the sequence of canonical maps
 \begin{equation}\label{ch01}0\longto V_{t}\longto M_t \longto \sO_{X}(1)
\longto 0.\end{equation} is a 
 short exact sequence of sheaves of $\sO_X$-modules

 In case  $\Char~k = p  >0$, the bundle $V_0$ has 
  the strong $\mu$-reduction bundle $V_{t_0}$, for some $t_0\leq t$
with the short exact sequence of $\sO_X$-sheaves
 \begin{equation}\label{chp}0\longto V_{t_0}\longto M_{t_0} \longto 
\sO_{X}(1) \longto 0.\end{equation}

\begin{rmk}\label{r11}Note 
that for a  given choice of   generators of $I$, the sequence~(\ref{ch0})
 and hence the bundles $V_t$ and $V_{t_0}$ are 
unique, but need not be unique for the pair $(R,I)$.
\end{rmk}

\subsection{The maximum support $\alpha(R, I)$ of the HK density function 
$f_{R,I}$}

\begin{thm}\label{n1}Following 
the Notations~\ref{n4}, if $(R, I)$ is a standard graded pair 
 over a perfect field of 
 characteristic $p>0$ and $V_{t_0}$ is a strong $\mu$-reduction bundle for 
$V_0$ then 
 \begin{enumerate}
\item $f_{R,I}(x) = f_{V_{t_0}, \sO_{X}(1)}(x)- f_{M_{t_0}, \sO_{X}(1)}(x)$, for $x\geq 1$. 
\item Moreover 
$$\alpha(R, I) := \mbox{Sup}~\{x\mid f_{R,I}(x)>0\} = 1- 
{a_{min}(V_{t_0})}/{d}.$$
\end{enumerate}
\end{thm}
\begin{proof}Both the assertions  follow from Lemma~\ref{hk2} and (\ref{fr}).
\end{proof}

\begin{rmk}\label{rrr}In the following two cases  
the bundle  $V_0$ itself is  the 
strong $\mu$-reduction bundle of $V_0$.
 \begin{enumerate}
\item  If $I$ has 
a set of generators of the same  degrees. Then 
$\mu_{min}(V_{0}) < \mu_{min}(M_0)$ and therefore  
$a_{min}(V_{0}) < a_{min}(M_0)$.

\item Suppose $h_1, \ldots, h_{\mu}$ is a set of minimal homogeneous generators 
of $I$. By Theorem~\ref{n1},  if $V_{t_0}\neq V_0$ is the  strong $\mu$-reduction 
 bundle then there is a graded ideal $J\subset I$ such that $I^* = J^*$, where
 $J$ is generated by a proper subset of the set $\{h_1, \ldots, h_\mu\}$.
Therefore if $I$ itself is the minimal graded {\em tight closure reduction} for 
$I$, {\em i.e.}, 
$$\{I\} = \mbox{min}\{J\subseteq I\mid J~\mbox{graded},~~J^* = I^*\}$$ then 
by choosing a  minimal generating set $\{h_1, \ldots, h_{\mu}\}$ in the 
short exact sequence  (\ref{ch0}), we can ensure that 
 $V_0$ itself is a strong $\mu$-reduction bundle. 
In particular, if $R$ is a $F$-regular ring then  $V_0 = V_{t_0}$.  
\end{enumerate}

In the following example we show that 
 $V_0$  
is not always  a strong $\mu$-reduction bundle of itself, which is equivalent to
 showing
$a_{min}(V_0) = a_{min}(M_0)$. Moreover, in the example,
 the functions $f_{M_0, \sO_X(1)}$
and $f_{V_0, \sO_X(1)}$ are  the same functions in the neighbourhood of 
their maximum common support.
In particular  
$\alpha(R, I) < 1-a_{min}(V_0)/d$.

\vspace{5pt}

\noindent{\bf Example}.
Let $R=k[x,y,z]/(x^d+y^d+z^d)$  and  
$I = (x^2, y^2, z^5)$.  Then,
by Lemma~3.2 of [S], 
$I$ is in the tight closure of $(x^2, y^2)$.
Hence $\alpha(R, I) = \alpha(R, (x^2, y^2)) = 4$, where the last equality follows 
by Theorem~4.10 of [TrW].

Now, for the pair 
$(R, I)$, the sequence (\ref{ch0}) is given by
$$0\longto V_0\longto M_0 = \sO_X(-1)\oplus \sO_X(-1)\oplus \sO_X(-4) \longto 
\sO_X(1)\longto 0$$
and the strong HN data of $M_0$ is $(\{-d, -4d\}, \{2, 1\})$ and $\mu(V) = -7d$. 

If  $a_{min}(V_0) \neq  a_{min}(M_0)$ then 
 $V_{0}$ is the strong $\mu$-reduction of $V_0$. Hence  
$a_{min}(V_{0}) = -3d$ which would imply the strong HN filtration of $V_0$ is
$0\subset \sL_1\subset F^{s*}V_0$, for some $s\geq 0$, where $\sL_1$ is a 
line bundle. But then 
 $$a_{max}(V_0) = \deg~\sL_1/p^s = \deg~V_0 - a_{min}(V_0) = 
-4d < a_{min}(V_0).$$
Hence $a_{min}(V_0) = a_{min}(M_0) = -4d$.
In particular the strong HN data of $V_0$ is 
$(\{-3d, -4d\}, \{1, 1\}$.
Now the HK density functions $f_{M_0, \sO_X(1)}$ and 
$f_{V_0, \sO_X(1)}$ can be written as follows:
$$f_{M_0, \sO_X(1)}(x)   = \begin{cases} 3d(3-x) & \mbox{if}\quad x< 2\\
 d(5-x) & \mbox{if}\quad 2\leq x < 5\\
 0 & \mbox{if}\quad 5\leq x, \end{cases}$$

$$f_{V_0, \sO_X(1)}(x)   = \begin{cases} 
d(9-2x) & \mbox{if}\quad x< 4\\
 d(5-x) & \mbox{if}\quad 4\leq x < 5\\
0 & \mbox{if}\quad 5\leq x.
\end{cases}$$

\end{rmk}

\begin{rmk}\label{re} We can give a bound on the strong HN slope 
$a_{min}(V_{t_0})$ in terms of 
the degrees of the generators $h_1, \ldots, h_{\mu}$ of $I$ as follows.

Let ${\tilde d_1} < {\tilde d_1} <\ldots <{\tilde d}_{l_1}$ be 
the  degrees of these generators (see Notations~\ref{mu}).  
If $V_0$ itself is the strong $\mu$-reduction bundle then
$a_{min}(V_{0})/d < a_{min}(M_{0})/d = 1-{\tilde d}_{l_1}$. Moreover, if 
 $V_0$ is not a strong $\mu$-reduction of itself, {\em i.e.}, if $t_0\geq 1$ 
then, by Remark~\ref{ree}, 
$$a_{min}(M_{t_0-1})/d = 1-{\tilde d}_{l_1-t_0+1}
\leq a_{min}(V_{t_0})/d  < a_{min}(M_{t_0})/d = 1-{\tilde d}_{l_1-t_0}.$$ 
\end{rmk}

\subsection{The $F$-threshold $c^I({\bf m})$  and 
 $\alpha(R, I)$ in  $\Char~p>0$} 

\vspace{5pt}
Here we prove $c^I({\bf m}) = \alpha(R, I)$.
This equality is known to hold when $R$ itself is a normal domain. Though we 
know $\alpha(R, I) = \alpha(S, IS)$, we can not deduce the equality by considering 
 the normalization of $R$ as we 
do  not know if
$c^I({\bf m}) = c^{IS}({\bf m}S)$.

Let $Y= \mbox{Proj}~R$ and let $\pi:X\longto Y$ be  the canonical map then, 
by construction, the sequence (\ref{ch0}) descends to 
the canonical  sequence 
$$0\longto W_0\longto N_0 =\oplus_{i=1}^\mu
\sO_{Y}(1-d_i)\longto \sO_{Y}(1)
\longto 0$$ of $\sO_Y$-modules.
In fact  the following lemma  implies that  the exact sequences 
(\ref{ch01})  and (\ref{chp}) also descend to similar
exact  sequences of sheaves of $\sO_Y$-modules. 

\begin{lemma}\label{cal}If for any $i <l_1$ the sequence 
\begin{equation}\label{chi}
0\longto V_i\longto M_i\longby{{\tilde g_i}} \sO_X(1) \longto 0\end{equation}
is exact then  it descends to a 
 short exact sequence
$$0\longto W_{i}\longto N_{i}\longto \sO_{Y}(1)
\longto 0$$
of $\sO_Y$-modules. In particular $V_i = \pi^*(W_i)$, where $W_i$ is a 
vector bundle on $Y$.
\end{lemma}
\begin{proof}By definition 
$M_{i} = \sum_j\sO_X(-n_j)$, where $n_j$ are nonnegative integers.
Let $N_{i} = \sum_i\sO_Y(-n_j)$. 
Then  the map ${\tilde g_i}:M_i \longto \sO_X(1)$ descends to the canonical map
 $g_{i}:N_i\longto \sO_Y(1)$.

We claim that the map $g_i$ is surjective: Otherwise there is a closed point $y\in Y$ such that 
the map $g_i:(N_i)_y \longto (\sO_Y(1))_y$ factors through the map 
${\bf m}_{Y,y}\into \sO_{Y,y} = (\sO_Y(1))_y$.
But then for any $x\in \pi^{-1}(y) \neq \phi$, the map 
${\tilde g_i}:(M_i)_x\longto (\sO_X(1))_x = \sO_{X,x}$ factors through 
${\bf m}_{X, x}\into \sO_{X,x}$, which contradicts the surjectivity of 
$M_i\longto \sO_X(1)$.

Now we have a short exact sequence 
$$0\longto W_i\longto N_i\longby{g_i} \sO_{Y}(1)
\longto 0$$
of $\sO_Y$-modules, which is locally split exact.
Hence 
$$0\longto \pi^*W_i\longto \pi^*N_i=M_i\longby{{\tilde g_i}} \pi^*\sO_{Y}(1)=\sO_X(1)
\longto 0$$
is an exact sequence of $\sO_X$-modules and therefore is the same as the 
 sequence~(\ref{chi}).\end{proof}

In the following lemma 
 $F^n_X:X\longto X$ denotes the $n^{th}$-iterated  Frobenius map on $X$ 
(ditto for $Y$). 
The sheaf $K$ is a $0$-dimensional coherent sheaf of $\sO_Y$-modules 
given by the  canonical exact sequence
\begin{equation}\label{calee}
0\longto \sO_Y\longto \pi_*\sO_X\longto K\longto 0.\end{equation}

\begin{lemma}\label{cal2}Let  $W$ be  a vector bundle on $Y$ 
and $V=\pi^*W$ then 
$$h^1(X, (F_X^{n*}V)(m)) \leq h^1(Y, (F_Y^{n*}W)(m))
\leq h^1(X, (F_X^{n*}V)(m)) + s\cdot h^0(Y,K),$$
for all $m, n\geq 0$, where $s = \rank~W$.
\end{lemma}
\begin{proof}Since $W$ is a locally free sheaf of $\sO_Y$-modules 
we have the induced short exact sequence of $\sO_Y$-modules
$$0\longto (F_Y^{n*}W)(m)\longto (F_Y^{n*}W)(m)\tensor
\pi_*\sO_X\longto (F_Y^{n*}W)(m)\tensor K\longto 0.$$
But  $(F_Y^{n*}W)(m)\tensor K = K^{\oplus s}$, which gives the exact 
sequence
$$\longto H^0(Y,K^{\oplus s}) \longto H^1(Y,(F_Y^{n*}W)(m))\longto 
H^1(Y, (F_Y^{n*}W)(m)\tensor \pi_*\sO_X)\longto 0.$$
By the  projection formula
$$(F_Y^{n*}W)(m)\tensor \pi_*\sO_X = \pi_*(\pi^*((F_Y^{n*}W)(m)) = 
\pi_*[(F_X^{n*}\pi^*W)(m)] = \pi_*((F_X^{n*}V)(m))$$
which implies
 $$h^1(Y, (F_Y^{n*}W)(m)\tensor \pi_*\sO_X) = 
h^1(Y, \pi_*((F_X^{n*}V)(m))) = h^1(X, (F_X^{n*}V)(m)).$$
\end{proof}

\begin{thm}\label{t3}If $(R, I)$ is a standard graded pair, where $R$ is a  
$2$-dimensional domain then 
$$\alpha(R, I) = c^I({\bf m}).$$
In particular $c^I({\bf m}) = c^{IS}({\bf m}S)$, where $R\longto S$ 
is a finite graded degree $0$ morphism of rings.
\end{thm}
\begin{proof} By Propostion~4.4 of [TrW], we have 
$\alpha(R, I) \leq c^I({\bf m})$.
We only need to prove that $c^I({\bf m}) \leq \alpha(R,I)$.
Let $x_0 = \alpha(R, I) = 1-a_{min}(V_{t_0})/d$.
By Lemma~\ref{cal}, the sequence~(\ref{chp}) descends to the short exact 
sequence 
$$0\longto W_{t_0} \longto N_{t_0} \longto \sO_Y(1)\longto 0$$
and $V_{t_0} = \pi^*W_{t_0}$.
If $N_{t_0}= \oplus_j\sO_Y(1-d_{1j})\longto \sO_Y(1)$ is  the multiplication 
map given by the elements 
$h_{11}, \ldots, h_{1a}\in I$ of degrees $d_{11}, \ldots, d_{1a}$, respectively,
 then for $q = p^n \gg 0$ and $m\in\N$
 we have 
$$0\longto (F^{n*}W_{t_0})(m-q) \longto \oplus_j\sO_Y(m-qd_{1j})\longto 
\sO_Y(m)\longto 0$$ 
and therefore for $J = (h_{11}, \ldots, h_{1a})$,
$$\ell(R/I^{[q]})_m\leq \ell(R/J^{[q]})_m \leq h^1(Y, F^{n*}W_{t_0}(m-q)).$$
Let $q_1 = p^{n_1}$ be such that the HN filtration of 
$F^{n_1*}V_{t_0}$ is the strong 
HN filtration.  Then, by Remark~\ref{*1}~(1),
$$h^1(X, F^{(n+n_1)*}V_{t_0}(m-qq_1)) = 0\quad\mbox{for}\quad 
m\geq (d-3)+ x_0qq_1.$$ 

By Lemma~\ref{cal2}, there is a constant  $C_0$ such that 
$h^1(Y, F^{n+n_1*}W_{t_0}(m-qq_1))\leq C_0$,  for every 
$m\geq (d-3)+ x_0qq_1$.  This implies
(see Proposition~4.6 of [TrW]) 
$$h^1(Y, F^{n+n_1*}W_{t_0}(m-qq_1)) = 0\quad\mbox{for}\quad m\geq C_0 + 
(d-3)+ x_0qq_1.$$
In particular $\ell(R/I^{[qq_1]})_m = 0$, in other words 
${\bf m}^{m} \subset I^{[qq_1]}$.
Now $$c^I({\bf m})\leq \lim_{q\to \infty}\frac{1}{qq_1}
\left[C_0 + (d-3)+ x_0qq_1\right] = x_0.$$

The second assertion follows as we have ($S$ is considered as an 
$R$-module here)
$$\alpha(S, I) \leq c^{IS}({\bf m}S) \leq c^I({\bf m}) = \alpha(R, I) = \alpha(S, I),$$
where the first inequality and the last equality follow from Proposition~4.4 of
[TrW] and Proposition~2.14 of [T2], respectively.
\end{proof}

\subsection{The $F$-threshold $c^I({\bf m})$ and $\alpha(R, I)$ in  
characteristic $0$}

\begin{notations}\label{ex5} In this section we consider the sequence
(\ref{ch0}), where $\Char~k = 0$. The bundle $V_t$ denotes
  the $\mu$-reduction bundle of $V_0$ and the filtration
\begin{equation}\label{**}\cdots \subset W_{l+1}\subset W_l \subset V_{t-1}\subset 
\cdots \subset V_0 \end{equation}
denotes the HN filtration of $V_0$.

Now, by Lemma~\ref{l01} 
\begin{enumerate}
\item $W_l = V_t$, if $\mu_{min}(V_t) > \mu_{min}(V_{t-1})$ and 
\item
$W_l\subset  V_t$, if $\mu_{min}(V_t) = \mu_{min}(V_{t-1})$.
\end{enumerate}
Moreover
\begin{enumerate}
\item ${V_0}/{V_{1}}\simeq {M_0}/{M_{1}},\ldots, 
{V_{t-1}}/{V_{t}}\simeq {M_{t-1}}/{M_{t}}$.
\end{enumerate}

For the notion of spread the reader can refer to subsection~6.3 of [TrW] 
(or [EGA]~[4] for details).
We choose a  finitely generated $\Z$-algebra $A\subset k$ such that 
 $(A, R_A, I_A, A)$, $(A, S_A, IS_A)$, $(A, X_A, \sO_{X_A}(1))$ and 
$(A, X_A, V_{0A})$ are  spreads 
for $(R, I)$, $(S,IS)$, $(X, \sO_X(1))$ and $(X, V_0)$, respectively.

Restricting to the fiber $X_s$, where $s\in \mbox{Spec}~A$ is a closed point, 
we have
  the following exact
 sequence of locally free sheaves of $\sO_{X_s}$-modules (where $X_s = 
X_A\tensor_A{\overline {k(s)}}$ and $V_0^s = V_{0A}\tensor_A{\overline {k(s)}}$). 
Let $p_s = \Char~k(s)$

\begin{equation}\label{e1}
0\longto V_0^s\longto \oplus_{i=1}^\mu\sO_{X_s}(1-d_i)\longto \sO_{X_s}(1)\longto 
0.\end{equation}
Since $V_i = \mbox{ker}(V_0\longto M_0/M_i)$, the sheaf
$V_{i_A}$ is the kernel of the map $V_{0A}\longto M_{0A}/M_{i_A}$
and hence $V_i^s := V_{iA}\tensor_A{\overline{k(s)}} = (V_0^s)_i = 
V_0^s \cap (M_i)_s$, that is 
$$ \mbox{\em the reduction mod}~~ p_s~~\mbox{of}~~V_i =   
(\mbox{\em the reduction mod}~~ p_s~~\mbox{of}~~V_0)\cap M_i.$$

As a consequence of the openness of the semistability property of sheaves ([Ma]), 
we can further  choose  $A$ such that the spread of the HN filtration of 
$V_0$ can be defined similarly.
In particular, there are spreads $(A, W_{iA})$ of $W_i$ such that 
for every $s\in \Spec~A$, the HN filtration of $V^s$ is
$$\cdots \subset W^s_{l+1}\subset W^s_l \subset V^s_{t-1}\subset 
\cdots \subset V^s_0 =V^s$$
and therefore the bundle $V_0^s$ has the $\mu$-reduction bundle $V_t^s$, 
where $t$ is indepedent of the point $s$ and  where the underlying sequence is
$$0\longto V^s_0=V^s\longto M^s_0 = M^s \longto \sO_{X_s}(1)\longto 0.$$
\end{notations}

We recall the following result (Lemma~1.8 and Lemma~1.16 from [T1]).

\begin{thm}\label{t1}If $W$ is a vector bundle on a nonsingular projective curve 
$X$ over a field of $\Char~ 0$. Then there is a spread 
$(A, X_A, W_A)$ of $(X, W)$  such that if $s$ is  a closed point in 
$\Spec~A$ and $p_s > 
4(\mbox{genus}~X)(\rank~W)^3$ then  
\begin{enumerate}\item  for every $m\geq 1$, 
 the HN filtration of  $F^{m*}(W^s)$ is a refinement of the $m^{th}$ Frobenius pull back of the 
HN filtration of $W^s$. This means,
if the HN filtration of $W^s$ is 
$0\subset E_1\subset E_2\subset \cdots \subset E_l\subset W^s$ 
then the HN filtration
of $F^{m*}(W^s)$ is of the form
$$0\subset E_{01}\subset \cdots \subset E_{0t_0} \subset F^{m*}E_1\subset
\cdots \subset F^{m*}E_i\subset E_{i1}\subset \cdots\hfill\hfill $$
$$\hfill\hfill \cdots \subset E_{it_i}\subset F^{m*}E_{i+1} 
\subset \cdots \subset F^{m*}W^s.$$
In particular, for each $i$,
the HN filtration of $F^{m*}(E_{i+1}/E_i)$ is 
$$ 0 \subset E_{i1}/F^{m*}E_i \subset \cdots \subset E_{it_i}/F^{m*}E_i 
\subset F^{m*}(E_{i+1}/E_i).$$
\item $$\lim_{p_s\to \infty} a_{min}(W^s) = \mu_{min}(W).$$
\end{enumerate}
\end{thm}

Now we proceed to give a well defined notion of 
 $\alpha(R, I)$ in characteristic~$0$.
\begin{lemma}\label{l*} 
We have a spread $A$ such that
\begin{enumerate}
\item $ \mu_{min}(V_{t-1}) < \mu_{min}(V_t)
  \implies  \alpha(R_s, I_s) = 1-a_{min}(V_{t}^s)/d,~~\forall~~s\in \mbox{maxSpec}(A)$, 
\item $\mu_{min}(V_{t-1}) = \mu_{min}(V_t)
 \implies \alpha(R_s, I_s) = 1-a_{min}(V_{t-1}^s)/d,~~\forall~~s\in 
\mbox{maxSpec}(A)$,
 \end{enumerate}
where $V_t^s$ is a {\em reduction mod} $p_s$ of $V_t$.
\end{lemma}
\begin{proof}We  choose a spread $A$  as in 
Notations~\ref{ex5} such that $p_s > 4(\mbox{genus}~X)(\rank~V_0)^3$, 
for every $s\in \mbox{maxSpec}~A$.

Recall that $W_l \subseteq V_t \subset V_{t-1}$.

 We fix a closed point  $s\in \Spec~A$ and let $F:X_s\longto X_s$ denote the Frobenius map.  
Let $m_1$ ($m_1$ may depend on $s$) be an integer 
such that both $F^{m_1*}(V_t^s)$
and $F^{m_1*}(V_{t-1}^s)$ 
have strong HN filtration. Let $V^s_{t_0}$  be 
 the strong $\mu$-reduction bundle of $V_0^s$. This means  
 $F^{m_1*}(V^s_{t_0})$  is the $\mu$-reduction bundle of 
$F^{m_1*}(V^s_0)$ and 
$t_0\leq t$.

\vspace{5pt}

\noindent{Case~(1)}\quad Let $\mu_{min}(V_{t-1}) < \mu_{min}(V_t)$.
 
Then $V_t = W_l$ and the HN filtration of $V_0$ is 
$\cdots \subset W_{l+1}\subset V_t\subset V_{t-1} \subset \cdots V_0$.

Now by Theorem~\ref{t1}~(1),
the HN filtration for $F^{m_1*}(V_0^s)$ is 
$$\cdots \subset F^{m_1*}(W_{l+1}^s) \subset \cdots \subset F^{m_1*}(V_t^s)
\subset F^{m_1*}(V_{t-1}^s)
\subset \cdots \subset F^{m_1*}(V_0^s),$$
  as  $V_i^s/V_{i+1}^s\simeq 
M_i^s/M_{i+1}^s$ is 
strongly semistable on $X_s$, for $i<t$.
Hence 
$$\mu_{min}(F^{m_1*}(V_{i}^s)) =  
\mu_{min}(F^{m_1*}(M_i^s)), ~~\mbox{for all}~~i<t.$$
In particular $t_0 \geq  t$ and therefore 
 $a_{min}(V_{t_0}^s) = a_{min}(V_t^s)$ which implies $\alpha(R, I) 
= 1-a_{min}(V_t^s)/t$.

\vspace{5pt}

\noindent{Case~(2)}\quad Let $\mu_{min}(V_{t-1}) = \mu_{min}(V_t)$. 

Then 
the HN filtration for $F^{m_1*}(V_0^s)$ is 
$$F^{m_1*}(W_l^s)\subset \cdots  \subset F^{m_1*}(V_{t-1}^s)
\subset F^{m_1*}(V_{t-2}^s)
\subset\cdots \subset F^{m_1*}(V_0^s).$$
 
Therefore, we have 
\begin{equation}\label{***}\mu_{min}(F^{m_1*}(V_i^s)) = 
\mu_{min}(F^{m_1*}(M_i^s)), ~~\mbox{for all}~~i<t-1.\end{equation}
Hence, $t_0 = t-1$ or $t_0 = t$.
 
If $t_0 = t-1$ then $\alpha(R_s, I_s) = 1-a_{min}(V_{t-1}^s)/d$.

If $t_0 = t$ then 
$ \mu_{min}(F^{m_1*}(V_t^s))\geq  \mu_{min}(F^{m_1*}(V_{t-1}^s))$ and 
$$\mu_{min}(F^{m_1*}(V_{t-1}^s))=\mu_{min}(F^{m_1*}(M_{t-1}^s)) = 
p^m\mu_{min}(M_{t-1}) = p^{m_1}\mu_{min}(V_{t-1}^s).$$
On the other hand $p^{m_1}\mu_{min}(V_{t-1}^s) = p^{m_1}\mu_{min}(V_t^s) 
\geq \mu_{min}(F^{m_1*}(V_t^s))$.          
Hence $$a_{min}(V_t^s) = a_{min}(V_{t-1}^s) \implies 
\alpha(R_s, I_s) =  
1-a_{min}(V_t^s)/d = 1-a_{min}(V_{t-1}^s)/d.$$
 \end{proof}

\begin{thm}\label{vb1}Let $(R, I)$ be a standard graded pair 
defined over a field of characteristic~$0$ with a spread 
$(A, R_A, I_A)$ as in Notations~\ref{ex5}. 
Let $V_t$ be the $\mu$-reduction bundle of $V_0$.
Let $s \in \mbox{Spec}(A)$ denote a closed point and $p_s = \Char~R_s$. Then
\begin{enumerate}
\item  for every $x\geq 0$, 
$f^{\infty}_{R, I}(x) := \lim_{p_s\to \infty}f_{R_s, I_s}(x)$ exists and
the function\linebreak $f^{\infty}_{R, I}:[0, \infty)\longto [0, 
\infty)$ is a continuous compactly supported  function 
such that 

 $$\alpha^{\infty}(R,I) := \mbox{Sup}~\{x\mid f^\infty_{R, I}(x) \neq 0\} =
1- \frac{\mu_{min}(V_t)}{d}.$$
\item $\lim_{p_s\to \infty}\alpha(R_s, I_s) = \alpha^{\infty}(R,I)$. 
\end{enumerate}
\end{thm}
\begin{proof}\noindent{(1)}\quad By (\ref{fr}), we have
  $f_{R_s, I_s}(x) = f_{{V^s}, \sO_{X_s}(1)}(x)
-f_{{M^s}, \sO_{X_s}(1)}(x)$. 
On the other hand, for a vector bundle $E$ on $X$ there is a  spread $(A, E_A)$ 
such that 
$$f^\infty_{E,\sO_{X}(1)}(x):=  
\lim_{p_s\to \infty}f_{E^s, \sO_{X_s}(1)}(x)$$ exists,
where the function $f^\infty_{E,\sO_{X}(1)}$ can be written  
 in terms of HN data of $E$ (see Remark~6.6 of [TrW]).
In particular we have a well defined function 
$$f^\infty(R,I)(x) := \lim_{p_s\to \infty}f_{R_s, I_s}(x)
= f^\infty_{V_0,\sO_{X}(1)}(x) - f^\infty_{M_0,\sO_{X}(1)}(x),$$
 where the functions $f^\infty_{V_0,\sO_{X}(1)}$ and 
$f^\infty_{M_0,\sO_{X}(1)}$ can be written in terms of their respective  
HN data.
Moreover, if $(\{\mu_1, \ldots, \mu_{k+1}\}, \{r_1, \ldots. r_{k+1}\})$ is the HN data for $V_t$ then 
$$f^\infty(R,I)(x) = -r_{k+1}(\mu_{k+1}+xd)\quad\mbox{for}\quad 
x\in  
\left(-{\mbox{min}(\mu_k, \mu_{min}(M_t))}/{d}, -{\mu_{min}(V_t)}/{d}\right)$$
and 
$f^\infty(R,I)(x) = 0$, for 
$x\in [-{\mu_{min}(V_t)}/{d},~~ \infty)$.

Hence $f^{\infty}_{R,I}:[0, \infty)\longto [0,\infty)$ is a compactly supported continuous function and
$\alpha^\infty(R,I) = 1- {\mu_{min}(V_t)}/{d}$.

\vspace{5pt}
\noindent{(2)}\quad 
By Theorem~\ref{t1}~(2)
$\lim_{p_s\to \infty}a_{min}(V_t^s) = \mu_{min}(V_t)$.
 If $\mu_{min}(V_{t-1}) <\mu_{min}(V_t)$ then, by Lemma~\ref{l*}, 
$$\lim_{p_s\to \infty}\alpha(R_s,I_s) = 
\lim_{p_s\to \infty}1-{a_{min}(V_t^s)}/{d} = 1-{\mu_{min}(V_t)}/{d}.$$
If 
$\mu_{min}(V_{t-1}) = \mu_{min}(V_t)$ then
$$\lim_{p_s\to \infty}\alpha(R_s,I_s)  = 
\lim_{p_s\to \infty}1-{a_{min}(V_{t-1}^s)}/{d} = 1-{\mu_{min}(V_{t-1})}/{d} = 
1-{\mu_{min}(V_t)}/{d}.$$
\end{proof}

\vspace{5pt}

\noindent{\underline {Proof of Theorem}~D}~:\quad
By Theorem~\ref{t3}, for $p_s>0$
 we have 
$\alpha(R_s,I_s) = c^{I_s}({\bf m}_s)$, hence assertion~(1) follows from Theorem~\ref{vb1}. 

(2) If $V_t$ is the $\mu$-reduction bundle of $V_0$ then $\mu_{min}(V_t) 
\geq \mu_{min}(V_{t-1})$.

If
$\mu_{min}(V_t) > \mu_{min}(V_{t-1})$ then
$$c^I_{\infty}({\bf m}) = 1-{\mu_{min}(V_t)}/{d} = 1-{\mu_{min}(V^s_t)}/{d} 
\leq 1-{a_{min}(V_t^s)}/{d} = c^{I_s}({\bf m}_s).$$
If $\mu_{min}(V_t) = \mu_{min}(V_{t-1})$ then
$$c^{I}_{\infty}({\bf m}) = 1-{\mu_{min}(V_{t-1})}/{d} 
= 1-{\mu_{min}(V^s_{t-1})}/{d} \leq 
 1-{a_{min}(V_{t-1}^s)}/{d} = c^{I_s}({\bf m}_s).$$

\vspace{5pt}

\noindent{(3)}\quad
Suppose $V_0$ is semistable. 
Let $V_t$ be the $\mu$-reduction bundle of $V_0$.

 \vspace{5pt}

\noindent{Case~(1)}\quad If $t=0$. Then 
$t_0 =0$. Hence $c^{I}_{\infty}({\bf m})= 1- \mu(V_0)/d$ and 
$c^{I_s}({\bf m}_s)  = 1-{a_{min}(V_0^s)}/{d}$.

 \vspace{5pt}

\noindent{Case~(2)}\quad If $t\geq 1$. Then $0 \neq  V_t \subset V_{t-1}$ 
and, by Lemma~\ref{l01}, the HN filtration of $V_0$ is 
$0\subseteq W_l \subset V_{t-1} \subseteq V_0$.
Hence $V_{t-1} = V_0$ and $W_l =0$. So $V_1$ is the $\mu$-reduction bundle 
of $V_0$
such that  $\mu_{min}(V_1) = \mu_{min}(V_0) = \mu(V_0)$.
Hence again  $c^{I}_{\infty}({\bf m}) = 1- \mu(V_0)/d$ and 
$c^{I_s}({\bf m}_s) = 1-{a_{min}(V_0^s)}/{d}$.
Therefore
$c^{I}_{\infty}({\bf m}) = c^{I_s}({\bf m}_s) \iff
\mu(V_0) = a_{min}(V_0^s) $ $\iff 
V^s_0\quad\mbox{is strongly semistable}.$.\hspace{50pt} $\Box$

\section{$F$-thresholds and reduction {\em mod} $p$}

\begin{lemma}\label{5.1}Let $V$ be a vector bundle of rank $r$ on
a nonsingular projective curve $X$ of genus $g$ over a  field of $\Char~p>0$. If
$p > \mbox{max}\{4(g-1)r^3, r!\}$  then
$$a_{min}(V) < \mu_{min}(V) \implies
\mu_{min}(V) = a_{min}(V) + {a}/{pb},$$ where $a$, $b$ are positive
integers such that $\mbox{g.c.d.}(a, p) = 1$ and $0< a/b \leq 4(g-1)(r-1)$.
\end{lemma}
\begin{proof}Let $m$ be an integer such
that $F^{m*}(V)$ achieves
the  strong HN filtration. Since $V$ is not 
strongly semistable, the integer  $m\geq 1$. 
By definition $a_{min}(V) =
\mu_{min}(F^{m*}V)/p^m$.
By Lemma~1.14 of [T1],
$$\mu_{min}(F^{m*}V)/p^m + C/p = \mu_{min}(V), ~~\mbox{where}~~
0 < C \leq 4(g-1)(r-1).$$
Note that  $\mu_{min}(F^{m*}V)$ and $\mu_{min}(V) \in \Z[1/r!]$. This implies
$C p^{m-1}(r!) \in \N$ and we can write
$$\mu_{min}(V) = a_{min}(V) + \frac{C p^{m-1} (r!)}{p^m (r!)} =
a_{min}(V) + \frac{a}{p b},$$
where $a$ $b$ are positive integers such that $\mbox{g.c.d.}(a, p) = 1$.
This proves the lemma.\end{proof}

\vspace{5pt}
\noindent{\underline {Proof of Theorem}~E~:\quad
By Theorem~D,  if
$c^{I_p}({\bf m}_p) \neq c^I_{\infty}({\bf m})$ then $c^{I_p}({\bf m}_p)
> c^I_{\infty}({\bf m})$. Let $X = \mbox{Proj}~S$ where 
$R\longto S$ is the normalization of $R$. By Lemma~\ref{l*}, there is a 
vector bundle $W$ ($W = V_t$ or $V_{t-1}$, where $V_t$, $V_{t-1}$ 
 are the bundles
 given  as in Notation~\ref{ex5}) on $X$
 such that, for $p_s\gg 0$,
$$c^I_{\infty}({\bf m}) = 1-\mu_{min}(W)/d\quad \mbox{and}\quad
c^{I_s}({\bf m}_s) = 1-a_{min}(W^s)/d)\quad \mbox{and}\quad 
\mu_{min}(W^s) = \mu_{min}(W),$$
where  $W_s$  denotes  
the  {\em reduction
mod} $p_s$ bundle of $W$ on $X_s$ and $d= \deg~X$. 

Therefore, by Lemma~\ref{5.1}, if $g$ denotes the genus of $X$ and $r+1$ is the
 number of minimal generator of $I$, then  we can write
$$c^{I_s}({\bf m}_s) = 1-\frac{\mu_{min}(W^s)}{d} + \frac{a}{p_sb},$$
where $a, b \in \Z_+$ and $0< a/b \leq 4(g-1)(r-1)$.

Since $\mu_{min}(W^s) = d_1/r_1$, where $d_1, r_1\in \Z_{+}$ such that
$r_1 \leq r$, the theorem follows for $p_s\gg 0$.
\hfill\hfill\hspace{50pt} $\Box$

\begin{rmk}\label{r3}  By the 
above theorem,  
if  $c^{{\bf m}_p}({\bf m}_p)\neq c^{\bf m}_{\infty}({\bf m})$ then
$p$ divides the denominator of
$c^{{\bf m}_p}({\bf m}_p)$.  However, the following example from [TrW] 
(Example~6.9) shows that  
  the denominator need not always be a power of $p$.

Let $R_p= k[x,y,z]/(h)$ be the Klein curve of degree $d\geq 17$ over a field of 
characteristic $p\geq d^2$. In other words 
  $h= x^{d-1}y+y^{d-1}z+z^{d-1}x$. If, in addition, 
$d$ is odd integer and $p\equiv\pm 2\pmod{(d^2-3d+3)}$ then
we know ({\em loc.cit.})
$$c^{{\bf m}_p}({\bf m}_p) = (3pd+d^2-9d+15)/2pd
\quad\mbox{and}\quad c^{\bf m}_{\infty}({\bf m}) = 3/2.$$

\end{rmk}

\end{document}